\documentclass{article}

\usepackage[latin1]{inputenc}
\usepackage{amsmath}
\usepackage{amssymb}
\usepackage{amsthm}
\usepackage{enumitem}
\usepackage{xspace}
\usepackage{amsfonts}
\usepackage{mathrsfs}
\usepackage{multicol}
\usepackage{prooftree}
\usepackage[all,2cell]{xy}
\CompileMatrices
\UseAllTwocells
\SilentMatrices
\usepackage{graphicx}
\usepackage{float}
\usepackage{fancybox}
\usepackage{hyperref}
\usepackage{enumitem}

\def \catC {\mathbb{C}}

\newcommand{\tr}[1]{\xrightarrow{#1}}    
\newcommand{\tl}[1]{\xleftarrow{#1}}    

\newcommand{\Cospan}[1]{\mathsf{Cospan}(#1)}
\def \PROP {\mathbf{PROP}} 
\def \PRO {\mathbf{PRO}} 

\def \F {\mf{F}}
\newcommand{\mf}{\mathbf}

\def \frob {\mf{Frob}}

\newcommand{\rewiring}[1]{\ulcorner #1 \urcorner}

\newcommand{\PERM}{\mathbf{P}}

\newcommand{\DPOstep}[1]{\rightsquigarrow_{#1}}

\newcommand\symNet{\lower3pt\hbox{$\includegraphics[width=20pt]{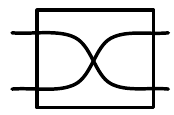}$}}
\newcommand\Idnet{\lower3pt\hbox{$\includegraphics[width=20pt]{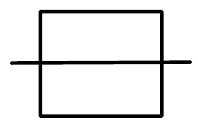}$}}

\newcommand \coc[1] {{#1}^{\star}}
\newcommand\lccB{\lower5pt\hbox{$\includegraphics[width=25pt]{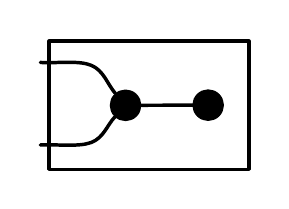}$}}
\newcommand\rccB{\lower5pt\hbox{$\includegraphics[width=25pt]{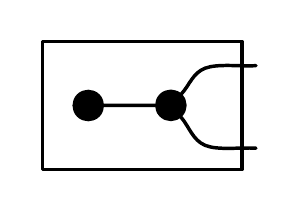}$}}
\newcommand\lccn{\lower5pt\hbox{$\includegraphics[width=20pt]{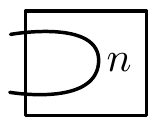}$}}
\newcommand\rccn{\lower5pt\hbox{$\includegraphics[width=20pt]{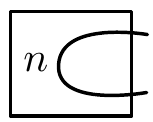}$}}
\def \df {\ \ensuremath{:\!\!=}\ }

\DeclareMathOperator{\id}{id}

\renewcommand{\P}{\mathcal{P}}

\newcommand{\Defeq}
 {\stackrel{\mathrm{def}}{=}}

\newcommand{\stran}{\raise1pt\hbox{$\centerdot$}}

\newcommand{\rring}[1]{\ensuremath{\mathbb{#1}}}

\newcommand{\N}{\rring{N}}

\newcommand{\Ra}{\Rightarrow}

\newcommand{\Set}{\cat{Set}\xspace}

\newcommand{\ladj}[2]{\ar@/^/[#1]^-{#2} \ar@{}[#1]|-%

{\ifthenelse{\equal{#1}{r}}{\bot}{%

{\ifthenelse{\equal{#1}{rr}}{\bot}{%

{\ifthenelse{\equal{#1}{l}}{\top}{%

{\ifthenelse{\equal{#1}{u}}{\dashv}{%

{\vdash}}}}}}}}}}

\newcommand{\radj}[2]{\ar@/^/[#1]^-{#2}}

\newcommand{\radjff}[2]{\ar@{_{(}->}[#1]^{#2}}

\newcommand{\pullbacktop}[4]{%

{#1} \ar@/_/[ddr]_{#4} \ar@/^/[drr]^{#2}%

\ar@{.>}[dr]|-{#3} \\}

\newtheorem{clm}{Claim}[section]

\theoremstyle{plain}

\newtheorem{pro}[clm]{Proposition}
\newtheorem{lem}[clm]{Lemma}
\newtheorem{thm}[clm]{Theorem}

\theoremstyle{definition}
\newtheorem{defn}[clm]{Definition}
\newtheorem{exm}[clm]{Example}

\newtheorem{rmk}[clm]{Remark}

\theoremstyle{remark}

\newcommand{\cat}[1]{\ensuremath{\mathbf{#1}}}

\newcounter{ncomm}

\newcommand{\ltsred}[1]
{ \setbox0=\hbox{$\ {}^{#1}\ $}
  \setbox1=\hbox{$\longrightarrow$}
  \loop\setbox1=\hbox{$-$\kern-0.3em\unhbox1}\ifdim\wd1<\wd0\repeat
  \hbox{$\ \ \mathop{\box1}\limits^{#1}\ \ $}
}

\newcommand{\arx}[2]{\!\xymatrix@=15pt{\ar[r]^{{#1}}_{{#2}}&}\!}

\newlength{\mylength}

{\setlength{\fboxsep}{15pt}
\setlength{\mylength}{\linewidth}%
\addtolength{\mylength}{-2\fboxsep}%
\addtolength{\mylength}{-2\fboxrule}%
\Sbox
\minipage{\mylength}%
\setlength{\abovedisplayskip}{0pt}%
\setlength{\belowdisplayskip}{0pt}%
}%
{\endminipage\endSbox
\[\fbox{\TheSbox}\]}

\newcommand{\syntax}[1]{\mathbf{S}_{#1}}

\newcommand{\FTerm}[1]{\mathbf{FTerm}_{#1}}
\newcommand{\Hyp}[1]{\mathbf{Hyp}_{#1}}

\newcommand{\BA}[1]{\ensuremath{\textbf{NB}_{#1}}\xspace}

\newcommand{\precU} {\ensuremath{\preceq_{U}}\xspace}
\newcommand{\precM} {\ensuremath{\preceq_{M}}\xspace}
\newcommand{\precL} {\ensuremath{\preceq_{L}}\xspace}
\newcommand{\precmu}{\ensuremath{\preceq_{\mu}}\xspace}
\newcommand{\precnu}{\ensuremath{\preceq_{\nu}}\xspace}

\usepackage{color}
\def\bR{\begin{color}{red}}
\def\bB{\begin{color}{blue}}
\def\bM{\begin{color}{magenta}}
\def\bC{\begin{color}{cyan}}
\def\bW{\begin{color}{white}}
\def\bBl{\begin{color}{black}}
\def\bG{\begin{color}{green}}
\def\bY{\begin{color}{yellow}}
\def\e{\end{color}\xspace}

\newcommand{\cgr}[2][scale=0.45]{\raisebox{0.1em}{\begingroup
\setbox0=\hbox{\includegraphics[#1]{graffles/#2}}%
\parbox{\wd0}{\box0}\endgroup}}

\def \poi {\,\ensuremath{;}\,}
\def \df {\ensuremath{:=}}
\def \tns {\ensuremath{\oplus}}
\def \: {\colon}

\newcommand{\fznote}[1]{\marginpar{{\bf F.Z.} #1 }}    

\newcommand{\X}{\mathbf{X}}
\newcommand{\Y}{\mathbf{Y}}

\newcommand\Bmult{\lower4pt\hbox{$\includegraphics[width=17pt]{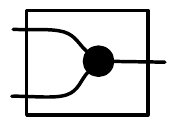}$}}
\newcommand\Bcomult{\lower4pt\hbox{$\includegraphics[width=17pt]{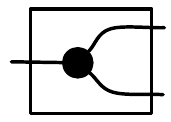}$}}
\newcommand\Bunit{\lower4pt\hbox{$\includegraphics[width=14pt]{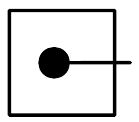}$}}
\newcommand\Bcounit{\lower4pt\hbox{$\includegraphics[width=14pt]{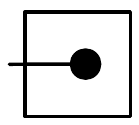}$}}

\newcommand\Wmult{\lower4pt\hbox{$\includegraphics[width=17pt]{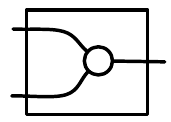}$}}
\newcommand\Wcomult{\lower4pt\hbox{$\includegraphics[width=17pt]{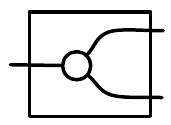}$}}
\newcommand\Wunit{\lower4pt\hbox{$\includegraphics[width=14pt]{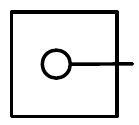}$}}
\newcommand\Wcounit{\lower4pt\hbox{$\includegraphics[width=14pt]{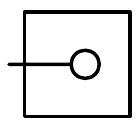}$}}

\newcommand{\SynToCsp}[1]{\ensuremath{\lfloor\!\lfloor{#1}\rfloor\!\rfloor}}    

\newcommand{\rrule}[2]{\ensuremath{\langle #1,#2 \rangle}}

\def \frob {\mf{Frob}}
\newcommand{\out}[1]{\mathsf{out}(#1)}
\newcommand{\inp}[1]{\mathsf{in}(#1)}
\newcommand{\lout}[1]{\mathsf{lout}(#1)}
\newcommand{\linp}[1]{\mathsf{lin}(#1)}
\newcommand{\rout}[1]{\mathsf{rout}(#1)}
\newcommand{\rinp}[1]{\mathsf{rin}(#1)}
\newcommand{\node}{\lower0pt\hbox{$\includegraphics[width=6pt]{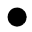}$}}
\newcommand{\hyperedge}{\lower2pt\hbox{$\includegraphics[width=25pt]{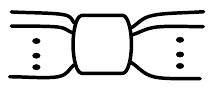}$}}
\newcommand{\ZeronetT}{\lower4pt\hbox{$\includegraphics[width=14pt]{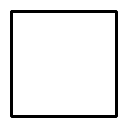}$}}
\def \Mon {\mathbf{M}}

\newcommand\idncircuit{\lower4pt\hbox{$\includegraphics[width=18pt]{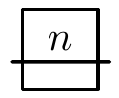}$}}
\newcommand{\rigidDPOstep}[1]{\Rrightarrow_{#1}}
\def \dfop {\ \ensuremath{=:}\ }
\def \NCB {\mathbf{NB} }

\usepackage{comment}

\title{Rewriting modulo symmetric monoidal structure}
\author{Filippo Bonchi, Fabio Gadducci, Aleks Kissinger \\
Pawel Sobocinski, Fabio Zanasi}

\begin{document}

\setlength{\pdfpageheight}{\paperheight}
\setlength{\pdfpagewidth}{\paperwidth}



\maketitle

\begin{abstract}
String diagrams are a powerful and intuitive graphical syntax for terms of symmetric monoidal categories (SMCs). They find many applications in computer science and are becoming increasingly relevant in other fields such as physics and control theory.

An important role in many such approaches is played by equational theories of diagrams, typically oriented and applied as rewrite rules. This paper lays a comprehensive foundation of this form of rewriting. We interpret diagrams combinatorially as typed hypergraphs and establish the precise correspondence between diagram rewriting modulo the laws of SMCs on the one hand and double pushout (DPO) rewriting of hypergraphs, subject to a soundness condition called convexity, on the other. This result rests on a more general characterisation theorem in which we show that typed hypergraph DPO rewriting amounts to diagram rewriting modulo the laws of SMCs with a chosen special Frobenius structure.

We illustrate our approach with a proof of termination for the theory of non-commutative bimonoids.
\end{abstract}

\section{Introduction}
\enlargethispage{\baselineskip}
\enlargethispage{\baselineskip}
\enlargethispage{\baselineskip}

 Symmetric monoidal categories (SMCs) are algebraic structures that feature sequential ($\poi$) and parallel ($\tns$) composition. This phenomenon seems commonplace, and indeed SMCs have found many applications in computer science, physics and related fields. Focussing on computer science, they have recently featured in concurrency theory, where they describe the concurrent nature of
 executions of Petri nets~\cite{Meseguer1990} as well as serving as their compositional algebra~\cite{Bruni2011, Sobocinski2014}, quantum information, where they have been used to model quantum circuits~\cite{Coecke2008, Coecke2012}, and in systems theory, where they provide a calculus of signal flow graphs~\cite{Bonchi2014b,Bonchi2015,BaezErbele-CategoriesInControl}. 

In each case, the algebra of SMCs gives us a syntax to talk about domain-specific artefacts. However,
 the two composition operations in an SMC are related by functoriality, and symmetries are natural: this imposes a non-trivial structural equality relation on terms from the outset---something that researchers in process algbera might refer to as a \emph{structural congruence}---that makes using ordinary tree-like syntax ineffectual. Functoriality, in particular, means that, given terms $A, B, C, D$ where $A$,$B$ and $C$,$D$ can be composed sequentially, we have that
\[ (A \tns C) \poi (B\oplus D)
=
(A \poi B) \tns (C \poi D).
\]
This equation is sometimes reffered to as the \emph{middle-four interchange}.
As a consequence, the syntax is intrinsically 2-dimensional, and for this reason diagrams---in this context often referred to as \emph{string diagrams}---are a more efficient representation for arrows of SMCs. Using string diagrams, both sides of the equation above are represented diagrammatically as
\[
\cgr[height=1cm]{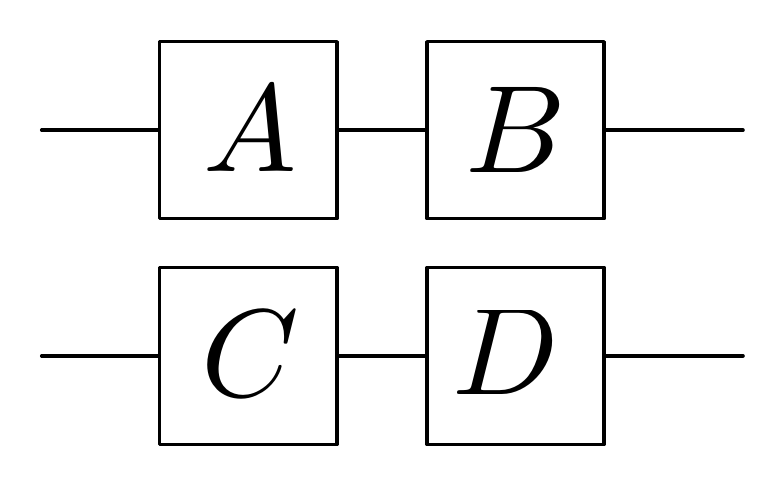}
\]
and so middle four interchange is built into the representation, along with other equational properties such as associativity of both composition operations.

String diagrams has been pioneered by Feynman and Penrose, but they remained just a tool for calculations, ultimately excluded from papers. This was likely due to a
lack of foundational
results that justified the soundness of their use: the careful mathematician checked each step
in a diagrammatic proof using standard term-based means.
%
%
%
This changed with the 1991 paper~\cite{Joyal1991} of Joyal and Street who formalised diagrams as topological structures and understood diagrammatic manipulation as homotopy.  Their framework allowed them to show that the resulting diagrams-up-to-homotopy-equivalence served as a description for the arrows of free braided monoidal categories, of which SMCs are a special case. Subsequently, the use of diagrammatic notation exploded, see e.g. the survey~\cite{Selinger2009}.
The results of Joyal and Street mean that we have a formal description of the nature of 2-dimensional syntax, and so of the arrows of free braided monoidal categories.

Most applications, however, do not feature free categories but rather rely on the presence of additional equations: for example, algebraic structures such as bimonoids and Frobenius monoids are commonplace.
Adding equations to a theory of string diagrams means that diagrammatic proofs include \emph{rewriting}: if the left hand side of an equation can be found in a larger string diagram, it can be deleted and replaced with its right hand side. Rewriting plays a similar role if string diagrams represent the topology of a system that may change dynamically during execution. From a mathematical point of view, one can formulate rewrite rules as generator 2-cells (this data structure is variously called a \emph{computad}~\cite{Street-computads} or a \emph{polygraph}~\cite{Burroni1993}) and consider the resulting free 2-category, where the 2-cells witness the possible rewriting trajectories. This does not solve the problem of how to \emph{implement} rewriting, and the approach of Joyal and Street does not offer an immediate solution either; we do not have an off-the-shelf rewriting theory for their diagrams.

One of the fundamental difficulties with working with terms modulo the laws of SMCs is finding matches. For example, consider the following rewrite rule
\[
\cgr[height=.7cm]{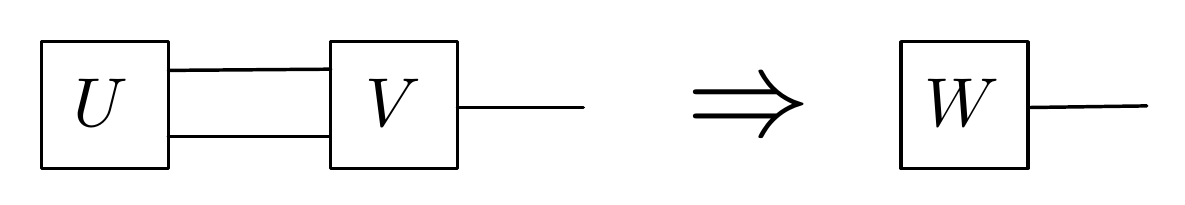}
\]
then, using naturality, we ought to be able to find a match in
\[
\cgr[height=1.7cm]{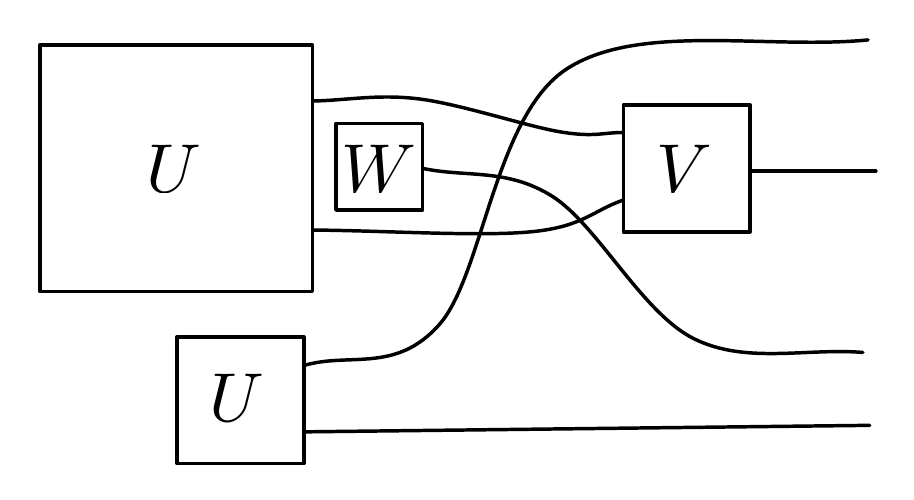}
\]
which, viewed as a term, does not contain the left hand side as a subterm, and requires a non-trivial deformation in order for this to be the case.

\medskip

Our approach is to think of string diagrams neither as terms nor as topological entities, but rather as combinatorial structures: i.e. as certain ``open'' (hyper-)graphs. Roughly speaking---since we are considering the symmetric monoidal case---the geometry of diagrams is discarded, and only connectivity information remains. Rewriting of diagrams then becomes an instance of graph rewriting, in particular the double pushout (DPO) approach, a well-trodden research topic that goes back to the 1970s~\cite{Ehrig1976}. The correspondence between rewriting modulo the laws of SMC and DPO is surprisingly tight, as we explain below.

To establish the connection between rewriting of string diagrams on the one hand and graph rewriting on the other, we extend and connect together two existing research threads. First, there has been work on understanding the algebraic nature of cospans of graphical structures e.g.~\cite{Rosebrugh2005}, where the algebra of special Frobenius monoids was shown to play a crucial role. In a cospan  $n \to G \leftarrow m$, the open nature of graphs is given by the `ìnterfaces'' $n$ and $m$ which, like ``dangling wires'' in string diagrams, allow composition on the left and on the right.

\noindent \begin{minipage}{0.85\linewidth}
Second, there has been work connecting computads in cospans~\cite{Sassone2005a,Gadducci1998} and DPO graph rewriting:
the key observation is that DPO
rules, which are usually presented as spans of graphs $L\leftarrow j \rightarrow R$, corresponds to rewrite rules in the cospan category as displayed on the right, where $0$ is the initial object in the category of graphs.
\end{minipage}
\begin{minipage}{0.15\linewidth}
\vspace{-0.5cm}
$$\xymatrix@C=0.02cm@R=0.3cm{ & j   \ar@/^/[dr] \ar@/_/[dl] & \\
L   \ar@{=>}[rr] & & R\\
& 0 \ar@/^/[ul] \ar@/_/[ur] &
} $$
\end{minipage}

Our starting point is the extension of the result in~\cite{Rosebrugh2005} from graphs to hypegraphs. This is essential to tackle arbitrary symmetric monoidal signatures $\Sigma$: the category of cospans of hypergraphs with type $\Sigma$ is isomorphic to the free \textit{hypergraph category}, that is the SMC freely generated by $\Sigma$ together with a chosen special Frobenius structure. Hypergraph categories have previously been called well-supported compact closed categories by Walters~\cite{Walters-lecture-WSCC}.

By connecting this result with the observations in~\cite{Sassone2005a,Gadducci1998}, we are able to show that DPO is just rewriting modulo the laws of SMCs ``plus'' Frobenius. More formally, a DPO system on hypergraphs of type $\Sigma$ amounts to a rewriting system on the free SMC generated by $\Sigma$ together with a chosen special Frobenius structure. Intuitively speaking, the special Frobenius structure brings a self-dual compact closed structure, which in turn allows us to bend rules around: an arbitrary rewriting rule of cospans, as below on the left, can always be transformed into the DPO rule on the right.
$$
\xymatrix@C=0.02cm@R=0.2cm{ & j   \ar@/^/[dr] \ar@/_/[dl] & \\
L   \ar@{=>}[rr] & & R\\
& i \ar@/^/[ul] \ar@/_/[ur] &
} \qquad \qquad \qquad
\xymatrix@C=0.001cm@R=0.2cm{ & i+ j   \ar@/^/[dr] \ar@/_/[dl] & \\
L   \ar@{=>}[rr] & & R\\
& 0 \ar@/^/[ul] \ar@/_/[ur] &
}
$$
Next we tackle rewriting modulo the equations of SMCs in isolation. We can still consider DPO, but we must be more careful about rule application: since we may not have a compact closed structure on the algebraic side, it may be the case that a graph rewriting rule is unsound when considered as a rewrite of string diagrams. We introduce a restricted form of DPO rewriting, called convex DPO rewriting, and we prove that it is sound and complete with respect to rewriting modulo symmetric monoidal structure.
\[
\hspace{-.5cm}\cgr[height=3.7cm,width=9.5cm]{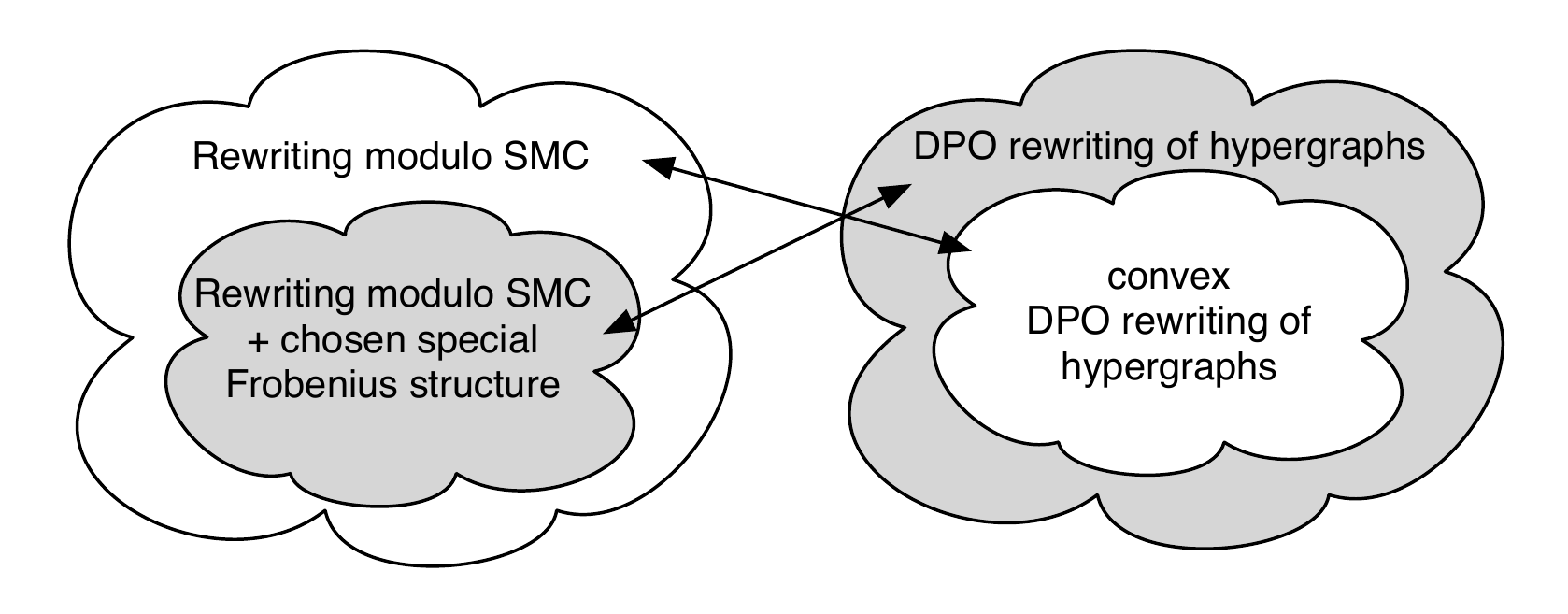}
\]
We conclude the paper with an illustrative example of our approach. We show termination for the theory of \emph{non-commutative bimonoids} (a.k.a. bialgebras). As any non-commutative group yields a bialgebra whose multiplication is non-commutative, these play an important role in representation theory, and the case where neither the multiplication nor the comultiplication is commutative forms the basis of the study of quantum groups~\cite{Kassel-quantumgroups2012}.

%

\paragraph{Related work.}

The correspondence between terms of algebraic structures with sequential and parallel composition (ultimately representing arrows of a free SMC) and flow diagrams (i.e., suitable hyper-graphs) has been recognized early on in computer science, and formalized at least since the work of Stefanescu (see the references in the survey~\cite{Selinger2009}). After~\cite{Joyal1991}, and especially after the paper on traced monoidal categories by Joyal, Street and Verity~\cite{Joyal_tracedcategories}, there has been a flourishing of interest in 2-dimensional diagrams. A 2-categorical equivalent for rewriting has also been devised almost immediately after, including the aforementioned correspondence with DPO rewriting~\cite{Gadducci1998}, even if usually for restricted kinds of diagrams. Our results duly extend those previous works.
On the one-side, the correspondence between string diagrams and SMCs (and their rewriting theories) is put on a theoretical firm
foundation, paving the way for a general theory encompassing different kind of diagrams, e.g. by enforcing constraints on the input and output degrees of nodes. On the other, it clearly recast the current development of PROP rewriting, establishing a fruitful link with the well-established, operational world of DPO rewriting.

A long tradition, pioneered by the work of Burroni on polygraphs~\cite{Burroni1993}, generalises term rewriting to higher dimensions, including the three-dimensional case of string diagram rewriting --- see e.g.~\cite{Mimram14} for a survey. In these works, the laws of SMCs are usually encoded as rewriting rules, resulting in rather elaborated rewriting systems whose analysis is often challenging (see e.g. \cite{Lafont2003,DBLP:journals/corr/abs-math-0612083}).
On the one hand, this traditional approach is far more general than ours that only deals with SMCs. On the other hand, our graphical representation allows to abstract away all the laws of SMCs, so to reduce considerably the number of rewriting rules and simplify the analysis of the rewriting systems.

In~\cite{open_graphs1}, the authors use cospans of \textit{string graphs} (a.k.a. open graphs) to encode morphisms in a symmetric monoidal category and reason equationally via DPO rewriting. There is an evident encoding of the hypergraphs we use to string graphs, which yields an equivalence between the category of directed cospans from~\cite{open_graphs1} and our subcategory of monogomous acyclic cospans in $\FTerm{\Sigma}$ (characterised in Theorem~\ref{thm:charactImage}). However, the notion of rewriting considered by those authors is more permissive, and is only sound in general when there is at least traced structure on the SMC.


\paragraph{Structure of the paper.} Section~\ref{sec:Background} provides some background on SMCs and the rewriting of the associated diagrams. Section~\ref{sec:HypInterpretation} introduces the combinatorial structures for interpreting diagrams and characterise them algebraically in terms of Frobenius monoids. Section~\ref{sec:dpoRewriting} properly establishes the correspondence between DPO rewriting and rewriting in a free SMC with a chosen special Frobenius monoid. Section~\ref{sec:SymMonDpoRewriting} characterises a restriction of DPO rewriting, called \emph{convex}, that is adequate for rewriting in any free SMC. Section~\ref{sec:terminationHopf} shows termination for the theory of non-commutative bimonoids. Appendix~\ref{app:proofs} contains the proofs omitted in the main text.

\section{Background}\label{sec:Background}

\paragraph{Notation.}
We write $f\poi g \: a \to c$ for composition of $f \: a \to b$ and $g\: b \to c$ in a category $\catC$.
For $\catC$ a symmetric monoidal category (SMC), $\tns$ is its monoidal product and $\sigma_{a,b} \: a \tns b \to b \tns a$ is the symmetry associated with $a,b \in \catC$. 
Given $\catC$ with pushouts, its cospan bicategory has the objects of $\catC$ as $0$-cells, cospans of arrows of $\catC$ as $1$-cells and cospan morphisms as $2$-cells. We denote with $\Cospan{\catC}$ the \emph{category} obtained by identifying the isomorphic $1$-cells and forgetting the $2$-cells.

\medskip

\begin{figure}[t]
 $$
\begin{array}{rcl}
(c \poi c') \poi d = c \poi (c' \poi d) & &
id_n \poi c = c = c\poi id_m\\
(c \tns c') \tns d = c \tns (c' \tns d) & &
id_0 \tns c = c  = c \tns id_0\\
\end{array} $$
$$
\begin{array}{c}
(c \poi c') \tns (d \poi d') = (c \tns d) \poi (c' \tns d')\\
\sigma_{1,1}\poi\sigma_{1,1}=id_2 \\
(c\tns id_z) \poi \sigma_{m,z} = \sigma_{n,z} \poi (id_z \tns c)
\end{array} $$
\caption{Laws of SMCs instantiated to a PROP $\X$.}\label{fig:axSMC}\end{figure}

\paragraph{SMTs and PROPs.} A standard way of expressing the algebraic structure borne in SMCs is through the notion of \emph{symmetric monoidal theory} (SMT). A one-sorted SMT is determined by $(\Sigma, E)$ where $\Sigma$ is the \emph{signature}: a set of \emph{generators} $o \: n\to m$ with \emph{arity} $n$ and \emph{coarity} $m$ where $m,n\in\N$. The set of $\Sigma$-terms is obtained by combining generators in $\Sigma$, the unit $\id \: 1\to 1$ and the symmetry $\sigma_{1,1} \: 2\to 2$ with $;$ and $\tns$. This is a purely formal process: given $\Sigma$-terms $t \: k\to l$, $u \: l\to m$, $v \: m\to n$, we construct new $\Sigma$-terms $t \poi u \: k\to m$ and $t \tns v \: k+n \to l+n$.  The set $E$ of \emph{equations} contains pairs of $\Sigma$-terms of the form $(t,t' \: i\to j)$; the only requirement is that $t$ and $t'$ have the same arity and coarity.

Just as regular (cartesian) algebraic theories have a categorical rendition as Lawvere categories~\cite{hyland2007category}, the corresponding (linear) notion for SMTs is the one of PROP~\cite{MacLane1965} (\textbf{pro}duct and \textbf{p}ermutation category). 
A PROP is a symmetric strict monoidal category with objects the natural numbers, where $\tns$ on objects is addition. Morphisms between PROPs are identity-on-objects strict symmetric monoidal functors. PROPs and their morphisms form a category $\PROP$.
Any SMT $(\Sigma,E)$ freely generates (presents) a PROP by letting the arrows $n\to m$ be the $\Sigma$-terms $n\to m$ modulo the laws of symmetric monoidal categories (in Figure~\ref{fig:axSMC}) and the (smallest congruence containing the) equations $t=t'$ for any $(t,t')\in E$.

When the set of equations $E$ is empty, we use $\syntax{\Sigma}$ to denote the PROP freely generated by $(\Sigma,E)$.
There is a natural graphical representation of these terms as string diagrams, which we now sketch referring to~\cite{Selinger2009} for the details. A $\Sigma$-term $n \to m$ is pictured as a box with $n$ ports on the left and $m$ ports on the right, to which we shall refer with top-bottom enumerations $1,\dots,n$ and $1,\dots,m$. Composition via $\poi$ and $\tns$ are rendered graphically by horizontal and vertical juxtaposition of boxes, respectively.
    \begin{eqnarray}\label{eq:graphlanguage}
t \poi s \text{ is drawn }
\lower5pt\hbox{$\includegraphics[height=.4cm]{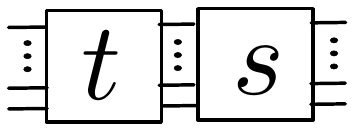}$}
\quad
t \tns s \text{ is drawn }
\lower10pt\hbox{$\includegraphics[height=.8cm]{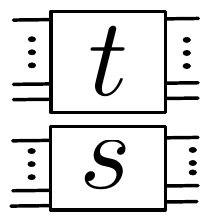}$}
\end{eqnarray}
    There are specific diagrams for the $\Sigma$-terms generating the underlying symmetric monoidal structure: these are $\id_1 \: 1 \to 1$, represented as $\Idnet$, the symmetry $\sigma_{1,1} \: 1+1 \to 1+1$, represented as $\symNet$, and the unit object for $\tns$, that is, $\id_0 \: 0 \to 0$, whose representation is an empty space $\ZeronetT$. Graphical representation for arbitrary identities $\id_n$ and symmetries $\sigma_{n,m}$ are generated according to the pasting rules in~\eqref{eq:graphlanguage}. Note that, of the equations displayed in Figure~\ref{fig:axSMC}, the first five are implicit in the diagrammatic language.

It will be sometimes convenient to represent $\id_n$ with the shorthand diagram $\idncircuit$ and, similarly, $t \: n \to m$ with $\lower5pt\hbox{$\includegraphics[height=15pt ]{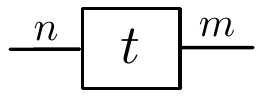}$}$.

\begin{exm}\label{exm:props}~
\begin{enumerate}[label=(\alph*)]
\item \label{ex:cmonoids} A basic example is the theory $(\Sigma_M,E_M)$ of \emph{commutative monoids}. The signature $\Sigma_M$ contains two generators:  \emph{multiplication} --- which we depict as the string diagram $\Bmult \: 2 \to 1$ --- and \emph{unit}, represented as $\Bunit \: 0 \to 1$.
 Equations in $E_M$ are given in the leftmost column of Figure~\ref{fig:EQFrobenius}: they assert commutativity, associativity and unitality. We call $\Mon$ the PROP freely generated by the SMT $(\Sigma_M,E_M)$. A useful observation is that we can give a concrete representation to $\Mon$: it is isomorphic to the PROP $\F$ with arrows $n \to m$ the functions $\{0,\dots, n-1\}\to \{0, \dots, m-1\}$, see e.g.~\cite{Lack2004a} for details.
\item \label{ex:frob} An example which will play a key role in our exposition is the theory $(\Sigma_F,E_F)$ of \emph{special Frobenius monoids}.
The signature $\Sigma_F$ contains four generators:  $\Bmult \: 2 \to 1$, $\Bunit \: 0 \to 1$, $\Bcomult \: 1\to 2$ and $\Bcounit \: 1\to 0$. 
Equations $E_F$ are displayed in Figure~\ref{fig:EQFrobenius}. They subsume the theory of commutative monoids, in the leftmost column.

Dually, the middle equations state that $\Bcomult$ and $\Bcounit$ form a cocommutative comonoid. Finally, the two rightmost equations describe the interaction between these two structures. 
We call $\frob$ the PROP freely generated by $(\Sigma_F,E_F)$. Just as $\Mon$, also $\frob$ enjoys a concrete representation: it is isomorphic to the PROP $\Cospan{\F}$ of cospans in $\F$~\cite{Lack2004a}. 
The isomorphism $\xi\colon \frob \to \Cospan{\F}$ is the unique PROP morphism defined as follows on the generators of $\Sigma_F$.
\begin{equation*}
\begin{array}{rclrcl}
\Bmult & \longmapsto & \left( 2\rightarrow 1 \leftarrow 1 \right) & \Bunit & \longmapsto & \left( 0\rightarrow 1\leftarrow 1 \right)\\
\Bcomult & \longmapsto & \left( 1\rightarrow 1 \leftarrow 2 \right) & \Bcounit & \longmapsto & \left( 1\rightarrow 1 \leftarrow 0\right)
\end{array}
\end{equation*}

\item \label{ex:ncbialgebras} The theory of \emph{non-commutative bimonoids} consists of a signature:
\[ \{\Wmult,\Wunit,\Wcomult,\Wcounit\}\]
and the following equations:
\begin{equation} \label{eq:ncbialgebralaws}
\lower55pt\hbox{$\includegraphics[height=3.7cm]{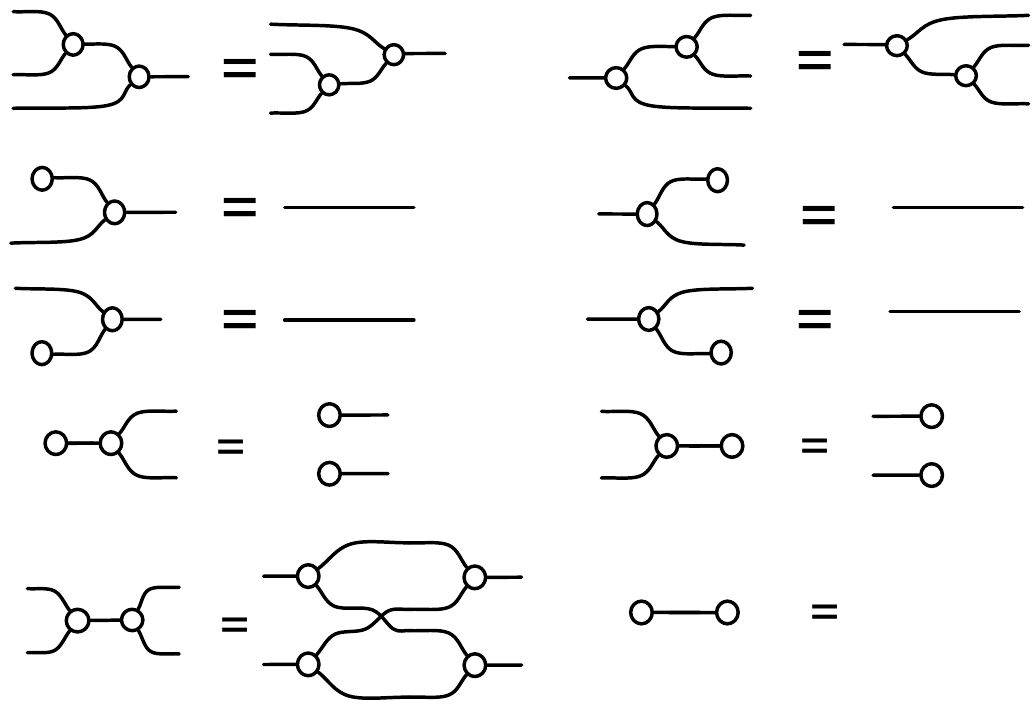}$}
\end{equation}
We call $\NCB$ the PROP freely generated by these data. Section~\ref{sec:terminationHopf} will show that the induced rewriting system terminates. For this purpose, it will be sometimes convenient to use labels $\mu,\eta,\nu,\epsilon$ resepctively to refer to the four generators of the signature.
\end{enumerate}
\end{exm}
 \begin{figure}[t]\label{fig:frob}
\[
\includegraphics[height=2cm]{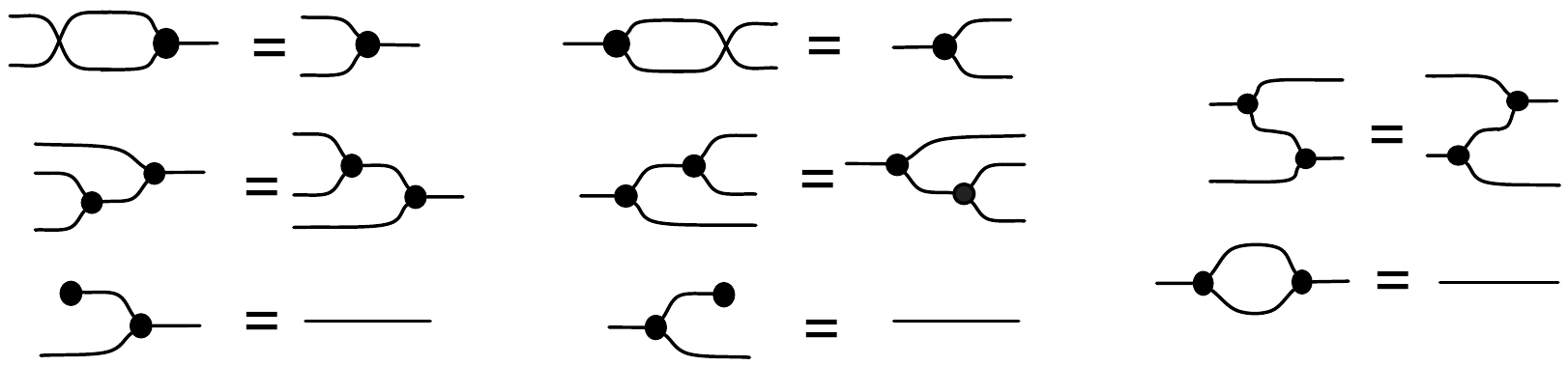}
\]
\caption{The equations $E_F$ of special Frobenius monoids.}\label{fig:EQFrobenius}
\end{figure}


\paragraph{Compact-closedness.} SMTs containing a Frobenius monoid will play a special role in our work. We say that an SMT $(\Sigma, E)$ is equipped with a \emph{chosen special Frobenius structure} if $\Sigma = \Sigma' + \Sigma_F$ and $E= E'+E_F$ for some signature $\Sigma'$ and set of equations $E'$. It is relevant for us that PROPs freely generated by such SMTs possess a self-dual compact closed structure~\cite{Kelly1980}. This amounts to the existence, for each $n$, of ``cups'' $\lccn \: n+n \to 0$ and ``caps'' $\rccn \: 0 \to n+n$, subject to the condition that
 \begin{align}\label{eq:snake}
 \lower9pt\hbox{$\includegraphics[height=.8cm]{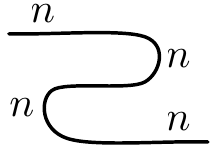}$} \ = \  \lower1pt\hbox{$\includegraphics[height=.25cm]{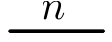}$}  \ = \  \lower9pt\hbox{$\includegraphics[height=.8cm]{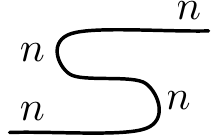}$}.
 \end{align}
 The equations of a Frobenius monoid allow to prove~\eqref{eq:snake} with cups defined according to the following pattern
 \begin{multicols}{2}
 \noindent
 \begin{eqnarray*}
\hspace{-.4cm}\lower8pt\hbox{$\includegraphics[width=18pt]{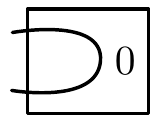}$} & \lower4pt\hbox{$\df$} & \quad  \lower4pt\hbox{$\ZeronetT$}
  \end{eqnarray*}
   \begin{eqnarray*}
  \lower0pt\hbox{$\includegraphics[width=18pt]{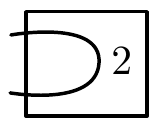}$} & \raise4pt\hbox{$\df$} & \lower6pt\hbox{$\includegraphics[height=.9cm]{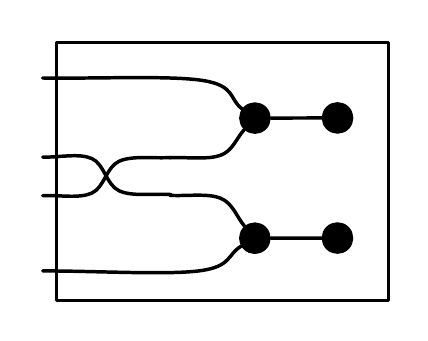}$}
 \end{eqnarray*}
   \begin{eqnarray*}
 \hspace{-.3cm}\lower4pt\hbox{$\includegraphics[width=18pt]{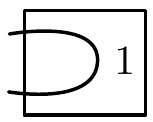}$} & \df & \ \ \lower7pt\hbox{$\includegraphics[height=.7cm]{graffles/rccr.pdf}$}
   \end{eqnarray*} \vspace{-.5cm}
   \begin{eqnarray*}
  \lower4pt\hbox{$\includegraphics[width=18pt]{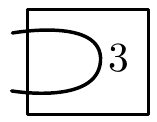}$} & \df & \lower12pt\hbox{$\includegraphics[height=1.1cm]{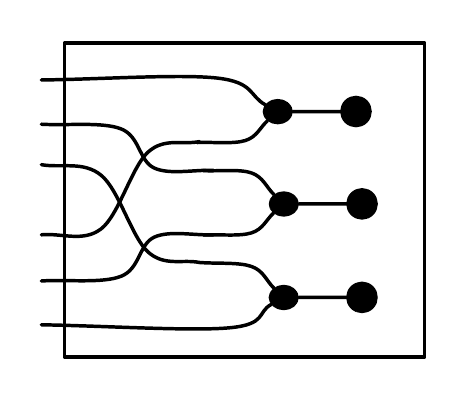}$}
   \end{eqnarray*}
 \end{multicols}
\noindent and similarly for caps, using instead $\Bcomult$ and $\Bunit$. 
The compact-closed structure also yields a contravariant PROP morphism $\coc{\cdot}$ (\emph{cf.} \cite[Rmk~2.1]{Selinger07DaggerCC}) defined as
\begin{align*}
 \lower8pt\hbox{$\includegraphics[height=.7cm]{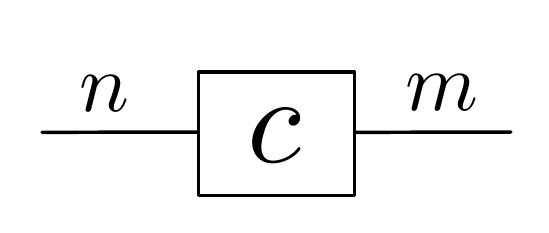}$} \ \ \ \ \mapsto \ \  \ \ \lower8pt\hbox{$\includegraphics[height=.7cm]{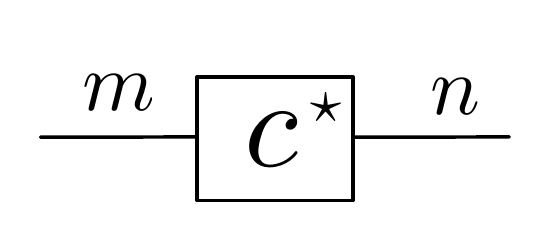}$} \ \df \  
 \lower10pt\hbox{$\includegraphics[height=.8cm]{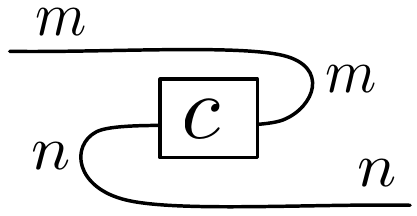}$}.
 \end{align*}

\medskip

\paragraph{Rewriting in a PROP.}  We now fix an arbitrary PROP $\X$. A rewriting \emph{rule} is a pair $\rrule{l}{r}$ where $l,r \: i \to j $ are arrows of $\X$ with the same source and target. We say that $i\to j$ is \emph{the type} of the rule $\rrule{l}{r}$, and sometimes write $\rrule{l}{r} \: (i,j)$. A \emph{rewriting system} $\mathcal{R}$ is a set of rewriting rules. The rewriting step $d \Ra_\mathcal{R} e$ is defined for any two arrows $d,e \: n \to m$ in $\X$ as   $d \Ra_{\mathcal{R}} e $ iff
$\exists \rrule{l}{r} \: (i,j) \in\mathcal{R}, c_1: n\to k+i, c_2: k+j \to n$ such that $d = c_1 \poi (id_k \tns r) \poi c_2$ and $e = c_1 \poi (id_k \tns r) \poi c_2$, i.e., diagrammatically:
\begin{eqnarray*}
 \cgr{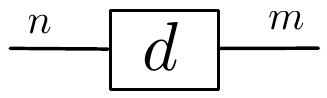} &=& \cgr[height=.75cm]{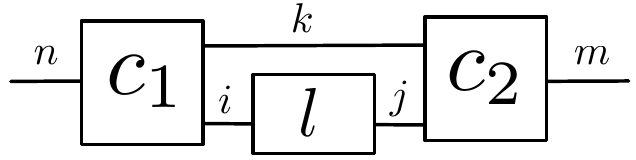} \\
 \cgr{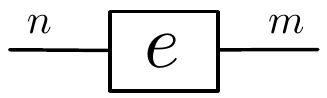} &=& \cgr[height=.75cm]{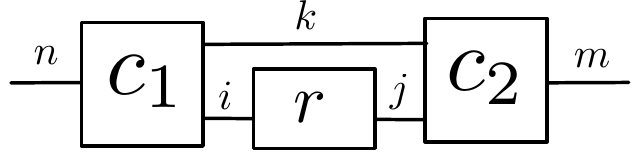}.
\end{eqnarray*}

When the PROP under consideration $\X$ is $\syntax{\Sigma}$ for some signature $\Sigma$, the above notion of rewriting step can be equivalently reformulated  as follows: in order to apply a rewrite rule $\langle l,r \rangle \: (i,j)$ to a diagram $d$ in $\syntax{\Sigma}$
we need to find a match of $l$ in $d$. This means that we need to find a \emph{context} $C$, namely a term in $\syntax{\Sigma + \{\bullet \: i \to j\}}$ with exactly one occurrence of $\bullet$, such that $d = C[l/\bullet]$.
The rewrite then takes $d\Ra_\mathcal{R} e$, where $e = C[r/\bullet]$.
%
%

\begin{rmk}
Another common way of viewing a rewrite system on (the arrows) a monoidal category is a to view a collection
of rules $\rrule{l}{r}:(i,j)$ as a computad~\cite{Street-computads} or polygraph~\cite{Burroni1993}. One can then consider the free 2-category on this data, in a similar way to how one obtains a free category on a directed graph.
\end{rmk}

With these definitions, diagrammatic reasoning can now be seen as a special case of rewriting. Given an arbitrary SMT $(\Sigma,E)$
we can obtain a rewriting system $\mathcal{R}_E$ as
\[
\mathcal{R}_E = \{\, \langle t,t'\rangle \,|\, (t,t') \in E \} \cup \{\, \langle t',t\rangle  \,|\, (t,t')\in E \,\}.
\]
\begin{pro}\label{pro:diagrammaticreasoning}
Let $c,d$ be two diagrams in $\syntax{\Sigma}$. Then $c = d$ in the PROP freely generated by $(\Sigma,E)$ iff
$c \Ra_{\mathcal{R}_E}^* d$.
\end{pro}




\section{Frobenius termgraphs}\label{sec:HypInterpretation}

In order to implement rewriting of string diagrams, we shift our focus away from viewing them as essentially a shorthand for terms, in favour of considering them as combinatorial hypergraph-like structures. In the following let $\F$ be the PROP of functions, introduced in Example~\ref{exm:props}.\ref{ex:cmonoids}. We start by defining the category of hypergraphs.
\begin{defn}[Hypergraphs]
The category of finite directed hypergraphs $\Hyp{}$
is the functor category $\F^{\mathbf{I}}$ where
${\mathbf{I}}$ has as objects pairs of natural numbers $(k,l) \in \N \times \N$ together with one extra object $\star$. 
 For each $k,l\in\N$, there are $k+l$ arrows from $(k,l)$ to $\star$.
\end{defn}

The idea is that an object $f$ of $\F^{\mathbf{I}}$ is the hypergraph with set of \emph{nodes} $f(\star)$
and, for each $k,l$, $f(k,l)$ the set of \emph{hyperedges} with $k$ (ordered) sources and $l$ (ordered) targets.
We shall visualise hypergraphs as follows: $\node$ is a node and $\hyperedge$ is an hyperedge, with tentacles attached to the left boundary linking to sources and the ones on the right linking to targets. Here is an example.
\[\lower12pt\hbox{$\includegraphics[height=1cm]{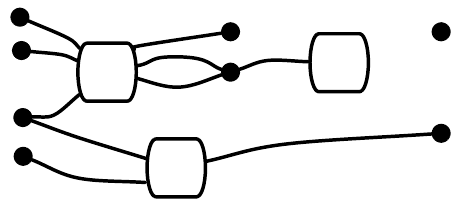}$}\]

In order to serve as an interpretation for string diagrams, we need to extend such vanilla hypergraphs in two ways, leading to the notion of \emph{Frobenius termgraph} (Def.~\ref{def:frobtermgraph}). The first is to assign hyperedges to the generators in a given signature. Formally, this is achieved as follows. Any symmetric monoidal signature $\Sigma$  can be considered as a hypergraph with a single node, in the obvious way. We can then express $\Sigma$-typed hypergraphs as the objects of the slice category $\Hyp{\Sigma} \Defeq \Hyp{}/\Sigma$. The typing is represented pictorially by labeling hyperedges with generators in $\Sigma$.
\[\lower12pt\hbox{$\includegraphics[height=1cm]{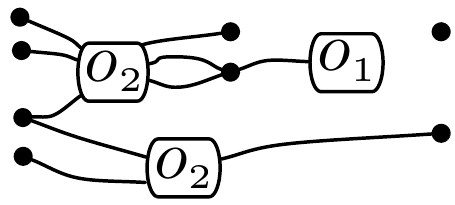}$}\]

Henceforward we shall only work in the slice, thus referring to its objects directly as ($\Sigma$-)hypergraphs.

The second piece of data needed to properly represent string diagrams is to identify the ``left and right dangling wires''. A natural solution is to use \emph{discrete cospans}, that is, the full subcategory of $\Cospan{\Hyp{\Sigma}}$ (the \emph{category} of cospans)\footnote{Since $\F$ has pushouts, so does $\Hyp{}$, where they are calculated pointwise in the usual way.} with objects discrete hypergraphs, i.e. those hypergraphs with no edges.

Such discrete cospans can encode terms, termgraphs~\cite{TermGraphSurvey,GadducciCorradini-termgraph} (where variable sharing can occur), but also cyclic structures where, for example, an output of an operation can be fed back as an input. As we shall see (Theorem~\ref{thm:coproduct}), there is a deep connection with the theory of special Frobenius monoids; for this reason, we will refer to these discrete cospans as \emph{Frobenius termgraphs}.

\begin{defn}[Frobenius termgraphs] \label{def:frobtermgraph}
The PROP $\FTerm{\Sigma}$ of $\Sigma$-\emph{Frobenius termgraphs}
is the full subcategory of $\Cospan{\Hyp{\Sigma}}$ with objects the discrete hypergraphs.
 Discrete hypergraphs are in 1-1 correspondence with the objects of $\mathbf{F}$, i.e. the natural numbers, whence $\FTerm{\Sigma}$ is a PROP for the same reasons as $\Cospan{\F}$.
\end{defn}

The terminology \emph{Frobenius termgraph} is justified by the following algebraic characterisation of $\FTerm{\Sigma}$.
\begin{thm}\label{thm:coproduct}
$\FTerm{\Sigma} \cong \syntax{\Sigma} + \frob$
\end{thm}
Theorem~\ref{thm:coproduct} is crucial for our exposition as it serves as a bridge between algebraic and combinatorial structures. It provides a presentation by generators and relation for the PROP $\FTerm{\Sigma}$, which is the union of the SMTs generating $\syntax{\Sigma}$ and $\frob$.

We remark that Theorem 3.3 generalises results already appearing in the literature (see e.g.
 ~\cite[Proposition 3.2]{Rosebrugh2005} and~\cite[Theorem 12]{Gadducci1998}, and the references therein)
 which however dealt only with operations of type $1\rightarrow 1$. Here we consider operations
with arbitrary (co)arities, thus drawing a tighter connection with PROPs.

%
\begin{proof}[Proof of Theorem~\ref{thm:coproduct}]
It suffices to verify that $\FTerm{\Sigma}$ satisfies the universal property of coproducts in $\PROP$. First we define morphisms
\[
\SynToCsp{\cdot} \: \mathbf{S}_\Sigma \to \FTerm{\Sigma} \text{ and }
\psi: \frob \to \FTerm{\Sigma}.\]
 Since $\syntax{\Sigma}$ is the PROP freely generated by an SMT with no equations, it suffices to define $\SynToCsp{\cdot}$ on the generators: for each $o \: n \to m$ in $\Sigma$, we let $\SynToCsp{o}$ be the following cospan of type $n \to m$.
\begin{equation}\label{eq:phi}
\cgr[height=1.3cm]{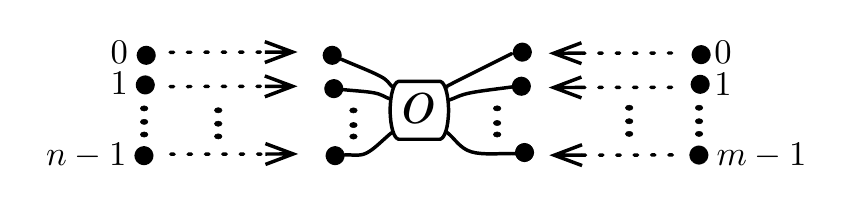}
\end{equation}
For the definition of $\psi$, recall the isomorphism $\frob \tr{\cong} \Cospan{\F}$ of Example~\ref{exm:props}.\ref{ex:frob}. We let $\psi$ be its composition with the trivial embedding $\Cospan{\F} \to \FTerm{\Sigma}$ given by regarding a set as a discrete graph.
Now, given a PROP $\mathbf{X}$ and morphisms $\alpha:\syntax{\Sigma}\to\mathbf{X}$, $\beta: \mathbf{Frob}\to \mathbf{X}$ we must show that there exists a unique morphism $\upsilon: \FTerm{\Sigma} \to \mathbf{X}$ making the following diagram commute.
\begin{equation}\label{eq:fterm}
\vcenter{\xymatrix@R=16pt{
{\syntax{\Sigma}} \ar[dr]_\alpha \ar[r]^-{\SynToCsp{\cdot}} &
{\FTerm{\Sigma}} \ar@{.>}[d]^<<<\upsilon
& \ar[l]_-{\psi} {\mathbf{Frob}} \ar[dl]^\beta \\
& {\mathbf{X}}
}}
\end{equation}
To do this, we decompose---in an essentially unique way---any Frobenius termgraph into an expression where all the basic constituents lie in the image of $\phi$ and $\psi$. This is possible because we are working with discrete cospans: all of the action happens on the nodes, that intuitively serve as coat-hangers for the edges.

Let  $n \xrightarrow{f} G \xleftarrow{g} n$
be in $\FTerm{\Sigma}$.
Then $G=(N,E,\tau)$ where $N$ is the ordinal of nodes,
$E=\bigcup_{k,l\in \N} E_{k,l}$ is a family of (typed) hyperedges
and $\tau \: E\to \Sigma$ is the type morphism.
Since $n,m$ are discrete, we actually have a cospan of \emph{functions}
$n\xrightarrow {f} N \xleftarrow{g} m$.

Given a hyperedge $e\in E$, $\SynToCsp{\tau e}$ is as illustrated in~\eqref{eq:phi}.
Define
\begin{equation}
\label{eq:decompE}
\tilde{n}\xrightarrow{i} \tilde{E} \xleftarrow{o} \tilde{m} \ \df \ \bigoplus_{e\in E} \SynToCsp{\tau e}.
\end{equation}
 Intuitively, $\tilde{E}$ is the collection of the hyperedges of $G$, but disconnected, and
 $\tilde{n}$ and $\tilde{m}$ are the finite ordinals of
all the inputs and outputs, concatenated. There are the induced canonical maps $j:\tilde{n}\to N$ and
$p:\tilde{m}\to N$ that send a ``disconnected'' input or output node to the corresponding node in $N$.
Now consider the following composition of three arrows in $\FTerm{\Sigma}$.
\begin{equation}\label{eq:decomp}
\vcenter{\xymatrix@R=11pt@C=10pt{
& N & & N \tns \tilde{E} & & N\\
n \ar[ur]^{f} & & \ar[ul]_{[id\ j]} N+\tilde{n} \ar[ur]_{id \tns i} & & \ar[ul]^{id \tns o} N + \tilde{m} \ar[ur]^{[id\ p]} & & \ar[ul]_g m
}}
\end{equation}

It is straightforward to compute the resulting pushouts in $\F$ and conclude that we obtain
$n \xrightarrow{f} G \xleftarrow{g} m$ as a result. Indeed, it suffices to focus on the nodes (see Fig. \ref{fig:bigpushout} in appendix). Since the constituents of~\eqref{eq:decomp} are all in the image of $\phi$ and $\psi$, this means that
we can define $\upsilon$ on $G$: this assignment is well-defined and moreover unique since the decomposition is unique up-to permutation of $\tilde{E}$, and $\mathbf{X}$ is symmetric monoidal.
\end{proof}

We are interested in $\FTerm{\Sigma}$ as a combinatorial universe for rewriting in $\syntax{\Sigma}$. In the remainder of this section we thus focus on the coproduct injection $\SynToCsp{\cdot} \: \mathbf{S}_\Sigma \to \FTerm{\Sigma}$. Our first observation is that $\SynToCsp{\cdot}$ is \emph{faithful}; we refer the reader to Appendix~\ref{app:proofs} for the proof, which relies on properties of coproducts in $\PROP$.

\begin{pro}\label{prop:faithful}
$\SynToCsp{\cdot} \: \syntax{\Sigma} \to \FTerm{\Sigma}$ is faithful.
\end{pro}


%

We conclude this section with a combinatorial characterisation of the image of $\SynToCsp{\cdot}$. A preliminary series of definitions introduces the relevant hypergraph notions: \emph{monogamicity} and \emph{acyclicity}.

\begin{defn}[Degree of a node]
The \textit{indegree} of a node $v$ in hypergraph $G$ is the number of pairs $(h,i)$ where $h$ is an hyperedge with $v$ as its $i$-th target. Similarly, the \textit{outdegree} of $v$ is the number of pairs $(h,j)$ where $h$ is an hyperedge with $v$ as its $j$-th source.
\end{defn}

\begin{defn}[Monogamicity] \label{def:monogamous} Given $m \tr{f} G \tl{g} n$ in $\FTerm{\Sigma}$, let $\inp{G}$ be the image of $f$ and $\out{G}$ the image of $g$. We say that $G$ is \emph{monogamous} if $f$ and $g$ are mono and, for all nodes $v$ of $G$,
\begin{align*}
\textrm{the indegree of $v$ is } & \begin{cases} 0 &\mbox{if } v \in \inp{G} \\
1 & \mbox{otherwise.} \end{cases} \\
\textrm{the outdegree of $v$ is } & \begin{cases} 0 &\mbox{if } v \in \out{G} \\
1 & \mbox{otherwise} \end{cases}
\end{align*}
\end{defn}
\begin{exm}
The following three cospans are \emph{not} monogamous.
\[1 \tr{} \cgr{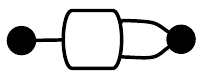} \tl{} 1 \quad 1 \tr{} \cgr{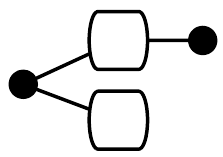} \tl{} 1 \quad 1 \tr{} {\raise2pt\hbox{$\cgr{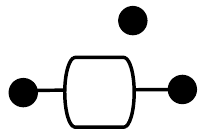}$}} \tl{} 1.\]
\end{exm}

\begin{rmk}
There is a more compact characterisation of monogamicity that uses the decomposition~\eqref{eq:decomp}: the factorised cospan is monogamous precisely when
\[
[ f \ p ] \: n+\tilde{m} \to V \text{ and }[ g\ j ] \: m+\tilde{n} \to V
\]
 are bijections.
\end{rmk}

%
%

The standard notion of \emph{directed path} from a node $v$ to a node $v'$ in a directed graph generalises to (directed) hypergraphs in the expected way. See Appendix~\ref{app:proofs} for the details.

\begin{defn}[Acyclicity] \label{def:directedacyclic}
    A hypergraph $G$ is \textit{directed acyclic} if there exists no directed path from a node to itself. We also call a cospan $n \tr{} G \tl{} m$ directed acyclic if the property holds for $G$.
\end{defn}

\begin{defn}[Convex sub-hypergraph] \label{def:convexsubhyp}
A sub-hypergraph $H \subseteq G$ is \textit{convex} if, for any nodes $v, v'$ in $H$ and any directed path $p$ from $v$ to $v'$, every hyperedge in $p$ must also be in $H$.
\end{defn}

\begin{lem}\label{lemma:convexfact}
    Let $m \rightarrow G \leftarrow n$ be a monogamous directed acyclic cospan and $L$ a convex sub-hypergraph. Then $L$ extends to a unique cospan $i \rightarrow L \leftarrow j$ such that $G$ factors as:
    \begin{equation}\label{eq:snd-factor}
    \left( n \tr{} C_1 \tl{} i\!+\!k \right) \poi\!\!\!\!\!
    \begin{array}{cc}
    \left( k \tr{id} k \tl{id} k \right)\\
    \tns \\
    \left( i \tr{} L \tl{} j \right)
    \end{array}
    \!\!\!\!\!\poi
    \left( j\!+\!k\tr{} C_2 \tl{} m \right)
    \end{equation}
    where all cospans in~\eqref{eq:snd-factor} are monogamous directed acyclic.
\end{lem}

\begin{thm}\label{thm:charactImage} $n \tr{} G \tl{} m$ in $\FTerm{\Sigma}$ is in the image of $\SynToCsp{\cdot}$ if and only if $n \tr{} G \tl{} m$ is monogamous directed acyclic. \end{thm}
\begin{proof}
The only if direction is easily verifiable by induction on $d$ such that $\SynToCsp{d} = n \tr{} G \tl{} m$. For the converse direction, we can reason by induction on the number of hyperedges in $G$. If $G$ does not contain any, then monogamicity and acyclicity imply that $n\tr{}G$ and $m\tr{}G$ are bijections, so that $n \tr{} G \tl{} m$ is in the image of an arrow only consisting of identities and symmetries.
Otherwise, suppose that $G$ consists of a single hyperedge $E$: it must have a $\Sigma$-type, say the generator $o \in \Sigma$. Because $n \tr{} G \tl{} m$ is monogamous directed acyclic, $o$ must have source $n$ and target $m$ as an arrow of $\syntax{\Sigma}$ and $\SynToCsp{o \: n \to m} = n \tr{} G \tl{} m$.

For the inductive step, we pick any hyperedge $E$ of $G$. By monogamicity and acyclicity, $E$ must be a convex sub-hypergraph of $G$. Hence, by Lemma~\ref{lemma:convexfact}, $n \tr{} G \tl{} m$ factors as~\eqref{eq:snd-factor}, with $L$ being $E$.
    The lemma guarantees that all the above cospans are monogamous directed acyclic. Therefore, by inductive hypothesis they are in the image of $\SynToCsp{\cdot}$, allowing us to conclude by functoriality that the same holds for $n \tr{} G \tl{} m$.
\end{proof}

\section{DPO rewriting on Frobenius termgraphs}\label{sec:dpoRewriting}


In this section, we exploit Theorem~\ref{thm:coproduct} to establish an equivalence between rewriting in a PROP and \emph{DPO rewriting}, defined below.

\begin{defn}[DPO rewriting]\label{defn:dpo}
A \emph{DPO rule} is a span $L \xleftarrow{} j \xrightarrow{} R$  in $\Hyp{\Sigma}$ where $j$ is a  discrete hypergraph.
A \emph{DPO rewriting system} $\mathcal{R}$ is a set of DPO rules.
Given $d,e\colon 0\to m$ in $\FTerm\Sigma$, namely $d = 0 \to G\stackrel{}{\leftarrow} m$ and $e = 0\to H\stackrel{}{\leftarrow} m$, we write that $d \DPOstep{\mathcal{R}} e$ if there exist a rule $L \xleftarrow{ } j  \xrightarrow{} R$ in $\mathcal{R}$ and a cospan
$j \rightarrow D \leftarrow m$ such that the following diagram in $\Hyp{\Sigma}$ commutes and the two squares are pushouts.
\begin{equation}\label{eq:dpo2}
\raise25pt\hbox{$
\xymatrix@R=10pt@C=20pt{
L \ar[d]   &  j \ar[d]
\ar@{}[dl]|(.8){\text{\large $\urcorner$}}
\ar@{}[dr]|(.8){\text{\large $\urcorner$}}
\ar[l] \ar[r]  & R \ar[d] \\
 G &  C \ar[l] \ar[r]  & H \\
&  m \ar[u] \ar[ur]_q  \ar[ul]^p
}$}
\end{equation}
The arrow $L \to G$ is called a \emph{matching} of $L$ in $G$.
\end{defn}
The above definition, already appearing in~\cite{Ehrig2004}, extends the standard one~\cite{Ehrig1976} with the ``interface'' $m$: the standard definition can be retrieved by simply taking $m$ to be the empty graph $0$. For examples, we refer the reader to Example~\ref{ex:unsoundcontext} and~\ref{ex:ncbialgebrasrewriting}.

\medskip

The first step in relating PROP and DPO rewriting consists in observing that the latter coincides with \emph{ground} PROP rewriting.
\begin{defn}\label{defn:ground}
A \emph{ground} rewriting system on a PROP $\X$ is a collection of
rules $\rrule{l}{r} \: (i,j)$ with $i = 0$.
\end{defn}
\begin{lem}\label{lemma:groundrewritingfact}
Let $\mathcal{R}$ be a ground rewriting system on a PROP $\X$ and $d,e: 0 \to m$ be in $\X$. Then
$d \Ra_{\mathcal{R}} e$ if and only if there are $\rrule{l}{r} \: (0,j)\in\mathcal{R}$ and $c:\ j \to m $ such that 
\begin{equation}\label{eq:rewriting}
  d =  l \poi c \quad \text{ and } \quad e = r  \poi c \text{.}
\end{equation}
\end{lem}
\begin{proof}
Suppose that $d \Ra_{\mathcal{R}} e$. By definition $d$ is like below left. It is equal to the diagram below right by the equations of SMCs. We define the context $c \: j \to m$ as the dashed diagram.
$$\includegraphics[height=1cm]{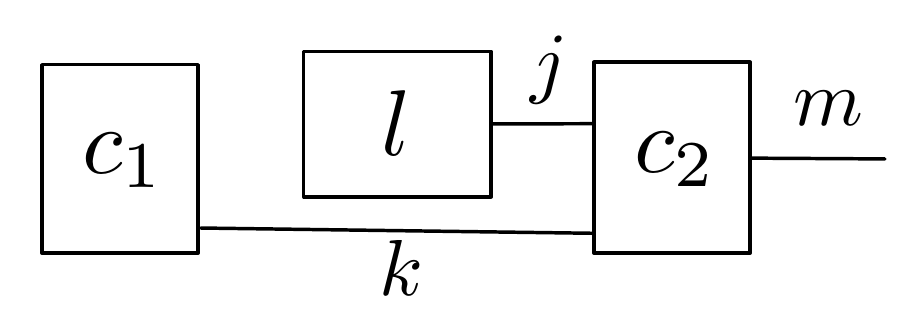} \qquad \qquad \qquad
\includegraphics[height=1.1cm]{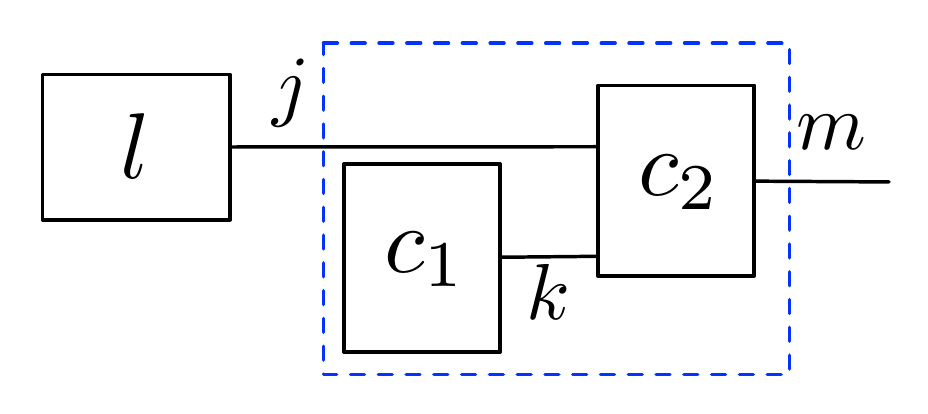}
$$
Similarly, we have $e = r \poi c$. The other direction is trivial.
\end{proof}


Observe that there is a 1-1 correspondence between ground systems in $\FTerm{\Sigma}$  and DPO systems. Indeed, any DPO rule
$L\xleftarrow{} j \xrightarrow{}R$
uniquely induces, by initiality of $0$ in $\Hyp{\Sigma}$, a ground rule $\rrule{0\rightarrow L \xleftarrow{}j}{0\rightarrow R\xleftarrow{}j}$, and viceversa. Moreover, this correspondence lifts to the rewrite relations.

\begin{thm}\label{thm:gadducciheckel}
Let $\mathcal{R}$ be a ground rewriting system in $\FTerm\Sigma$.
Then,
$ d \Ra_{\mathcal{R}} e$ iff $ d \DPOstep{\mathcal{R}} e$.
\end{thm}
\begin{proof}
The conditions in~\eqref{eq:dpo2} are the same as those on \eqref{eq:rewriting}: it is enough to take $c$ in \eqref{eq:rewriting} as $j \rightarrow C \leftarrow m$ in~\eqref{eq:dpo2} and recall that composition in $\FTerm{\Sigma}$ is defined by pushout.
\end{proof}

This correspondence can be further extended to arbitrary -- not necessarily  ground -- rewriting systems. Indeed, in PROPs freely generated by an SMT equipped with a chosen Frobenius structure, for any rewriting system there exists an equivalent ground one. The idea is to bend  a rule $\langle l, r \rangle$ of type $n \to m$ into a ground rule of type $0 \to n+m$ by exploiting the compact closed structure introduced in Section \ref{sec:Background}. To make this intuition formal, for any morphism $c\: n \to m$, we define $\rewiring{d} \: 0 \to n+m$ as
\[
\cgr[height=.9cm]{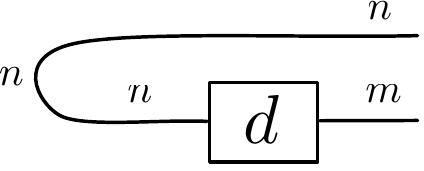}.
\]
Then, for a system $\mathcal{R}$, we define the ground rewrite system $\rewiring{\mathcal{R}}$ as
$$\{ \rrule{\rewiring{l}}{\rewiring{r}} \mid \rrule{l}{r}  \in \mathcal{R} \}$$ which, as stated by the following result, is equivalent to $\mathcal{R}$.
\begin{lem}\label{lem:frobeniusground}
Let $\mathcal{R}$ be a rewriting system on a PROP freely generated by an SMT equipped with a special Frobenius structure. Then,
$d\Rightarrow_{\mathcal{R}} e$ iff $\rewiring{d}\Rightarrow_{\rewiring{\mathcal{R}}} \rewiring{e}$.
\end{lem}
\begin{proof} For the left-to-right direction, let $\rrule{l}{r} \: i \to j$ be the applied rule. That means, for some $c_1$, $c_2$ and $k$,
\begin{eqnarray}
  \cgr{circuitd.pdf} & =& \cgr[height=.75cm]{factl.pdf} \nonumber \\
 \cgr{circuite.pdf} & =& \cgr[height=.75cm]{factr.pdf}. \label{eq:facteasr}
\end{eqnarray}
Now, using the compact-closed structure and the laws of SMCs we can transform $\rewiring{d}$ as
\begin{eqnarray}
  \cgr{rewd1.pdf} & =& \cgr{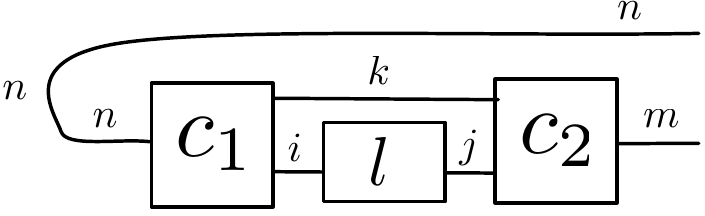} \nonumber\\
  & =& \cgr{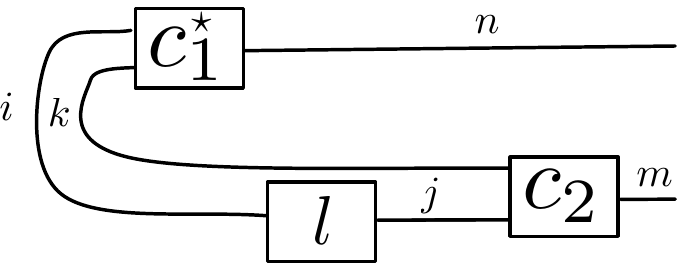} \label{eq:rewd}\\
  & =& \cgr{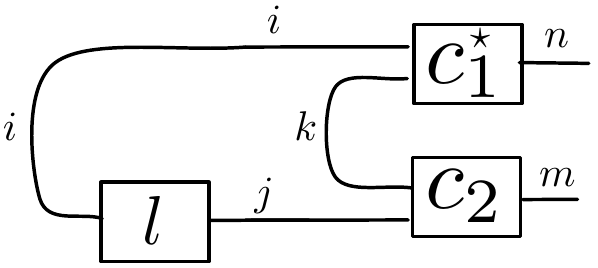}. \nonumber
\end{eqnarray}
There is an analogous computation for $\rewiring{e}$ exploiting~\eqref{eq:facteasr}.
\begin{eqnarray}
  \cgr{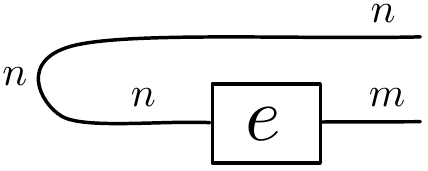} & =& \cgr{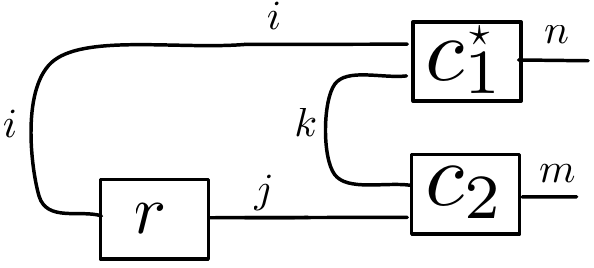} \label{eq:rewe}
\end{eqnarray}
To conclude, observe that the outcome of \eqref{eq:rewd} rewrites into the outcome of \eqref{eq:rewe} with the rule $\rrule{\rewiring{l}}{\rewiring{r}} \: i \to j$.

For the converse direction, suppose that $\rewiring{d}\Rightarrow_{\rewiring{\mathcal{R}}} \rewiring{e}$ via $\rrule{\rewiring{l}}{\rewiring{r}} \: i \to j$. By Lemma~\ref{lemma:groundrewritingfact} this means that, for some $c$,
\begin{eqnarray}
  \cgr{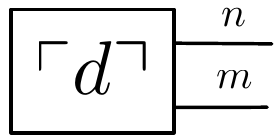} & =& \cgr{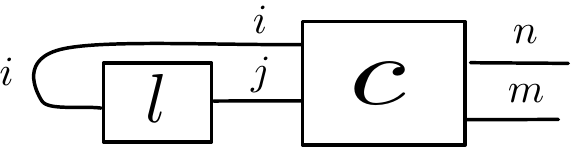} \label{eq:groundfactdasl}\\
 \cgr{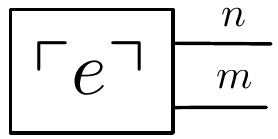} & =& \cgr{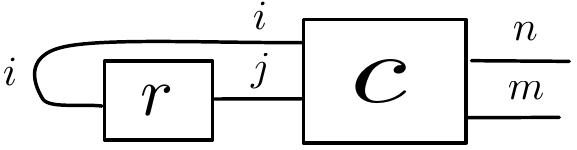}. \label{eq:groundfacteasr}
\end{eqnarray}

Now, by~\eqref{eq:groundfactdasl} we can suitably shape $d$ as follows using the compact-closed and symmetric monoidal structure.
\begin{eqnarray}
  \cgr{circuitd.pdf} & =& \cgr[height=.9cm]{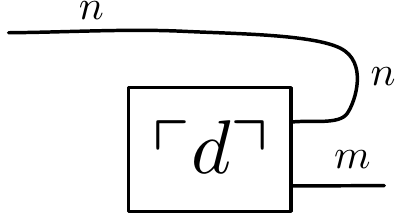} \nonumber\\
  & =& \cgr{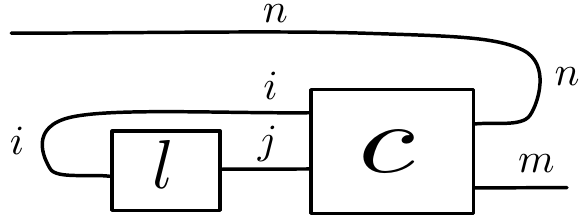} \label{eq:undorewd}\\
  & =& \cgr{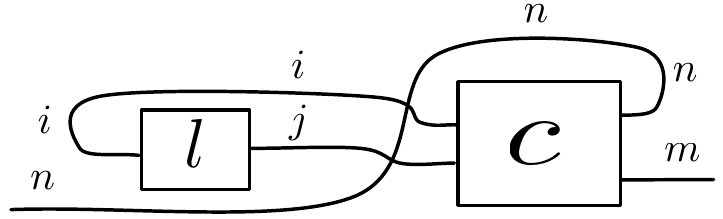} \nonumber
\end{eqnarray}
We act analogously on $e$ exploiting~\eqref{eq:groundfacteasr}.
\begin{eqnarray}
  \cgr{circuite.pdf} & =& \cgr{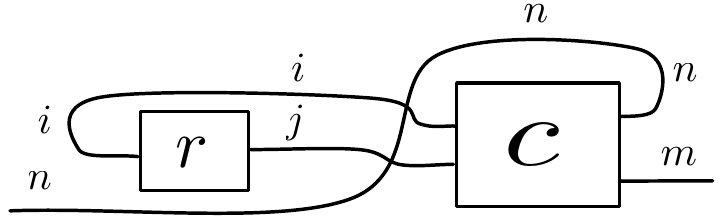} \label{eq:undorewe}
\end{eqnarray}
Note that the outcome of \eqref{eq:undorewd} is of shape $d_1 \poi (\id \tns r) \poi d_2$ and the one of \eqref{eq:undorewe} is of shape $d_1 \poi (\id \tns r) \poi d_2$, whence $d\Rightarrow_{\mathcal{R}} e$ by application of rule $\rrule{l}{r}$.
\end{proof}


We now combine the insights provided by Theorem~\ref{thm:gadducciheckel}
and Lemma~\ref{lem:frobeniusground} into
the following theorem, which gives a tight correspondence between rewriting in $\syntax{\Sigma}+\frob$ and DPO rewriting in $\FTerm{\Sigma}$. First, given a rewriting system $\mathcal{R}$ on
$\syntax{\Sigma}+\frob$, we define the system $\Phi(\mathcal{R})$ on $\FTerm{\Sigma}$ as $\{ \rrule{\Phi(l)}{\Phi(r)} \mid \rrule{l}{r}  \in \mathcal{R} \}$ where $\Phi \colon \syntax{\Sigma}+\frob \to \FTerm{\Sigma}$ is the iso of Theorem \ref{thm:coproduct}.

\begin{thm}\label{thm:frobeniusrewriting}
Let $\mathcal{R}$ be any rewriting system on $\syntax{\Sigma}+\frob$. Then,
\[
d \Rightarrow_\mathcal{R} e  \quad \text{ iff } \quad \Phi(\rewiring{d}) \DPOstep{\Phi(\rewiring{\mathcal{R}})}  \Phi(\rewiring{d})\text{ .}
\]
\end{thm}
\begin{proof}
The proof follows immediately by Lemma~\ref{lem:frobeniusground} and Theorems \ref{thm:coproduct},~\ref{thm:gadducciheckel}.
\end{proof}


A wide corpus of theorems, algorithm and tools has been developed in the last decade for DPO rewriting on \emph{adhesive} categories \cite{Lack2005}. These can be reused to deal for rewriting SMTs equipped with a Frobenius structure because of Theorem~\ref{thm:frobeniusrewriting} and the following.
\begin{pro}
$\Hyp{\Sigma}$ is adhesive.
\end{pro}
\begin{proof}
The PROP $\F$ of functions is adhesive for the same reasons that $\Set$ is.
Adhesivity of $\Hyp{\Sigma}$ then follows since it is a functor category over $\F$, see~\cite[Prop. 3.5]{Lack2005}.
\end{proof}

The theory of adhesive categories says that pushout complements for a DPO rewriting as in~\eqref{eq:dpo2} are uniquely defined when the arrow $j \to L$ is mono. However, this situation is not enforced by our approach, as shown by the following example.

\begin{exm}\label{ex:unsoundcontext}
Consider $\Sigma = \{ e_1 \: 0 \to 1, e_2 \: 1\to 0, e_3 \: 1 \to 1\}$ and a rewriting system $\mathcal{R}$ with the rule $\rrule{\cgr{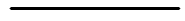}}{\cgr{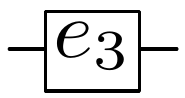}} \: (1,1).$ It gets interpreted in $\FTerm{\Sigma}$ as the DPO rule
$$\cgr{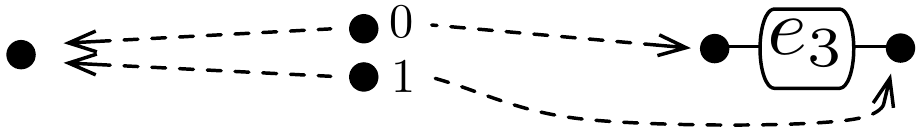}.$$
The left leg of the span is not mono, and therefore pushout complements are not necessarily unique for the application of this rule, as testified by the following two DPO rewriting steps.
\begin{equation}\label{eq:multiplecontext-goodoutcome}
\vcenter{
\xymatrix@=15pt{
{\cgr[width=7pt]{node.pdf}} \ar[d] & \ar@{}[dl]|(.75){\text{\large $\urcorner$}} {\cgr{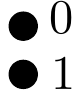}} \ar[d]_{g} \ar[l] \ar[r] \ar@{}[dr]|(.75){\text{\large $\ulcorner$}} & {\cgr[height=13pt]{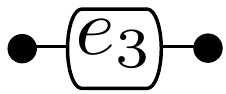}} \ar[d] \\
{\cgr[height=13pt]{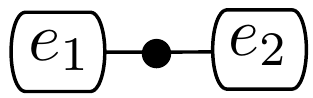}} & \ar[l] {\cgr[height=13pt]{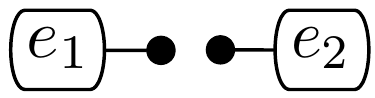}} \ar[r] & {\cgr[height=13pt]{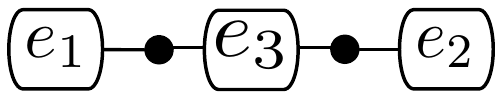}}
}
}
\end{equation}
\begin{equation}\label{eq:multiplecontext-badoutcome}
\vcenter{
\xymatrix@=15pt{
{\cgr[width=7pt]{node.pdf}} \ar[d]   & \ar@{}[dl]|(.75){\text{\large $\urcorner$}} {\cgr[height=13pt]{wrongcontext-twodots.pdf}} \ar[d]_{g'} \ar[l] \ar[r] \ar@{}[dr]|(.75){\text{\large $\ulcorner$}} & {\cgr[height=13pt]{wrongcontext-R.pdf}} \ar[d] \\
{\cgr[height=13pt]{wrongcontext-G.pdf}} & \ar[l] {\cgr[height=13pt]{wrongcontext-PC.pdf}} \ar[r] & {\cgr[height=27pt]{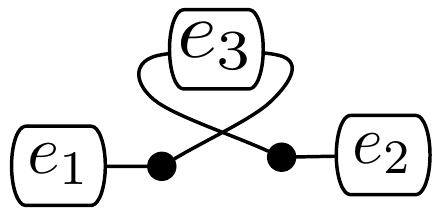}}
}
}
\end{equation}
The different outcome is due to the fact $g$ maps $0$ to the leftmost and $1$ to the rightmost node, whereas $g'$ swaps the assignments. Note that, as guaranteed by Theorem~\ref{thm:frobeniusrewriting}, both rewriting steps can be mimicked at the syntactic level in $\syntax{\Sigma} + \frob$. In the second case, one needs to use the compact closed structure.
\begin{eqnarray*}
&\cgr[height=13pt]{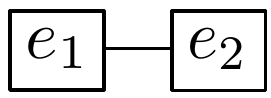} \qquad \qquad \Rightarrow_\mathcal{R}& \cgr[height=13pt]{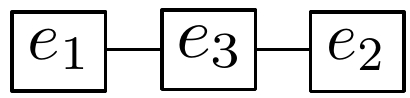} \\ \\
&\cgr[height=13pt]{multiplecontext1.pdf}= \cgr[height=22pt]{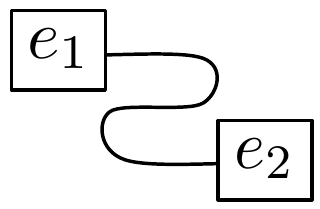}= \cgr[height=20pt]{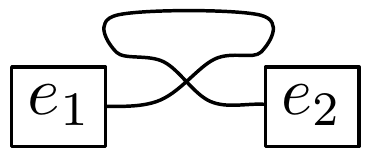} \ \ \ \Rightarrow_\mathcal{R}& \cgr[height=25pt]{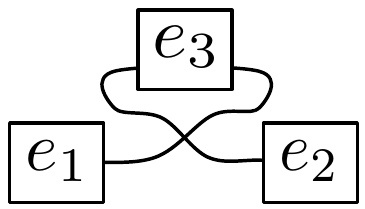}
\end{eqnarray*}
 \end{exm}

 \begin{exm}\label{ex:ncbialgebrasrewriting}
In a similar fashion we would like to study the rewriting theory of the SMTs introduced in Example \ref{exm:props}, in particular the case of non-commutative bialgebras (Ex.\ref{exm:props}.\ref{ex:ncbialgebras}). First, we orient left-to-right the equations~\eqref{eq:ncbialgebralaws}, creating a rewriting system $\mathcal{R}_{\scriptscriptstyle \NCB}$. We then obtain the ground rewriting system $\rewiring{\mathcal{R}_{\scriptscriptstyle \NCB}}$ by bending each rule in $\mathcal{R}_{\scriptscriptstyle \NCB}$ and, finally, via the translation $\Phi \: \syntax{\Sigma}+\frob \tr{\cong} \FTerm{\Sigma}$ we obtain the DPO system  $\Phi(\rewiring{\mathcal{R}_{\scriptscriptstyle \NCB}})$ --- see Figure \ref{fig:hopf}. In Section \ref{sec:terminationHopf}, we will show termination of $\Phi(\rewiring{\mathcal{R}_{\scriptscriptstyle \NCB}})$. However, for this to work we first need a rewriting procedure that is sound and complete for rewriting in SMTs, like $\NCB$, that miss the extra Frobenius structure required in Theorem~\ref{thm:frobeniusrewriting}. This extension of our approach is the content of the next section.
\end{exm}

\section{The general symmetric monoidal case}\label{sec:SymMonDpoRewriting}


In the general case, where there is no chosen special Frobenius structure, we can still use $\FTerm{\Sigma}$ for rewriting, but we need to careful about how we find \emph{legal}
contexts. For instance, the context in~\eqref{eq:multiplecontext-badoutcome} relies on the extra Frobenius structure being present.  

%
%


The remit of this section is to tailor a suitable restriction of DPO rewriting that adheres to rewriting for an SMT, just as generalised DPO rewriting corresponds to rewriting for an SMT with a chosen Frobenius structure (Theorem~\ref{thm:frobeniusrewriting}).

Before the formal developments, let us sketch the intuition. Given a rewriting system $\mathcal{R}$, we are interested in the shape of contexts for $\rewiring{\mathcal{R}}$-rewriting in $\syntax{\Sigma}+\frob \cong \FTerm{\Sigma}$ that ``look like'' legal contexts for $\mathcal{R}$-rewriting in $\syntax{\Sigma}$. By Lemma~\ref{lem:frobeniusground}, a generic context in $\syntax{\Sigma}+\frob$ for a rule $\rrule{\rewiring{l}}{\rewiring{r}}$ has shape
\[ \cgr{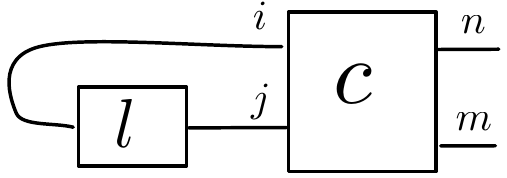} \]
Now, this context corresponds to a legal one in $\syntax{\Sigma}$ only when we can find $c_1$, $c_2$ in $\syntax{\Sigma}$ such that
\begin{equation}\label{eq:goodcontext} \cgr{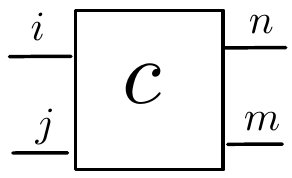} = \cgr{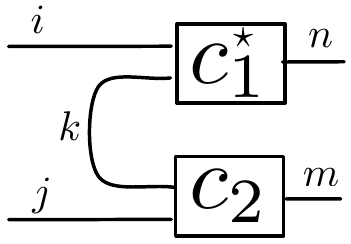}.\end{equation}
Indeed, in that case, we are in presence of the bent version of a legal context in $\syntax{\Sigma}$ for rule $\rrule{l}{r}$ (cf. \eqref{eq:rewd}):
\[ \cgr{rewd4.pdf} \ =\  \cgr{rewd2.pdf}\]


Our task is to characterise legal contexts like~\eqref{eq:goodcontext}. We proceed in two steps. First, we introduce the notion of \emph{boundary complement} (Definition~\ref{def:boundarycomplement}), which rests on monogamicity (Definition~\ref{def:monogamous}). It clarifies the role that interfaces $i,j,m,n$ play in the legal context~\eqref{eq:goodcontext}: $j$, $n$ serve as inputs and $i$, $m$ serve as outputs of its un-bent version. Contrary to arbitrary contexts in $\FTerm{\Sigma}$ (\emph{cf.} Example~\ref{ex:unsoundcontext}), boundary complements are uniquely defined up-to iso when they exist (Proposition~\ref{thm:uniquenessBoundaryCompl}).
The second notion is \emph{convex} matching (Definition~\ref{def:convexmatching}),
which enforces that the only communication between $c_1$ and $c_2$ happens along the interface $k$, \emph{cf}~\eqref{eq:goodcontext}.

A convex matching always allows to pick a (unique) boundary complement as context: we call this situation a \emph{convex} a DPO rewriting step (Definition~\ref{def:rigidDpoRewriting}). The main result of the section is the adequacy of convex DPO rewriting for rewriting in an SMT (Theorem~\ref{th:adequacyRigidSMT}).

\begin{defn}[Boundary complement] \label{def:boundarycomplement}
    For monogamous cospans $i \xrightarrow{a_1} L \xleftarrow{a_2} j$ and $n \xrightarrow{b_1} G \xleftarrow{b_2} m$,
    and a mono $f : L \to G$, a pushout complement as $(\dagger)$ below
    \[ \xymatrix@C=50pt@R=20pt{
        L \ar[d]_f \ar@{}[dr]|{(\dagger)} & {i + j} \ar[l]_{a := [a_1,a_2]}
              \ar[d]^{c := [c_1,c_2]}
        \ar@{}[dl]|(.8){\text{\large $\urcorner$}}
         \\
       G & L^\perp \ar[l]^{g} \\
       & n+m \ar[ul]^{[b_1,b_2]} \ar@{-->}[u]_{[d_1,d_2]}
       } \]
    is called a \textit{boundary complement} if $c$ is mono and there exists $d_1\: n\to L^\perp$ and $d_2\: m \to L^\perp$, making the above triangle commute, such that
    \begin{equation}\label{eq:boundary} j+n \xrightarrow{[c_2,d_1]} L^\perp \xleftarrow{[c_1,d_2]} m+i \end{equation}
    is a monogamous cospan.
\end{defn}

%

\begin{pro}\label{thm:uniquenessBoundaryCompl}
    Boundary complements in $\Hyp{\Sigma}$ are unique, when they exist.
\end{pro}

As mentioned, the notion of boundary complement does not tell the whole story about which matches in $\FTerm{\Sigma}$ are legal in $\syntax{\Sigma}$. The missing requirement is that $c_1^{\star}$ and $c_2$ in~\eqref{eq:goodcontext} can only communicate through interface $k$, which connects outputs of $c_1$ to inputs of $c_2$. The following counterexample shows that, even in presence of a boundary complement, linking outputs of $c_2$ to inputs of $c_1$ may still yield an illegal rewriting step in $\syntax{\Sigma}$.

\begin{exm}\label{ex:unsound} Fix $\Sigma = \{ e_1 \: 2 \to 1, {e_2 \: 1 \to 2}, {e_3 \: 1 \to 1} , \\ e_4 \: 1 \to 1\}$ and consider the following rewriting rule in $\syntax{\Sigma}$.
\begin{eqnarray}
         \rrule{\lower9pt\hbox{$\includegraphics[height=.6cm]{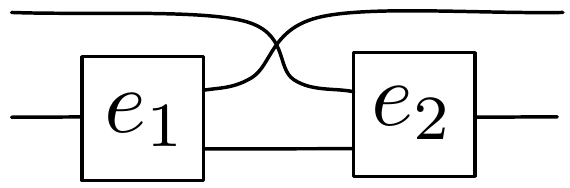}$}\ \ }{\lower9pt\hbox{\ \ $\includegraphics[height=.7cm]{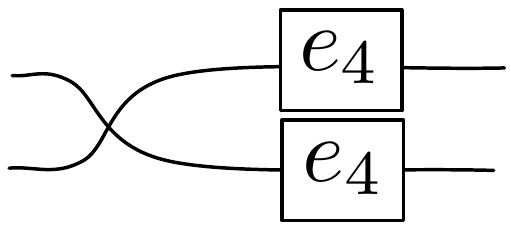}$}} \: (2,2) \label{eq:counterexSoundness1} \end{eqnarray}
         Left and right side are interpreted in $\FTerm{\Sigma}$ as cospans
        \begin{eqnarray*}
         2 \tr{} \cgr[height=.6cm]{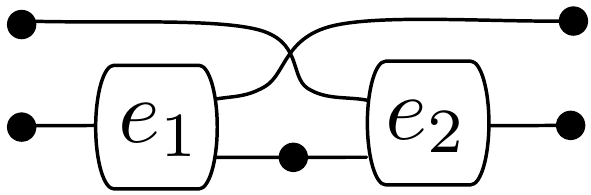} \tl{} 2 & \!\!\!\!\! & 2 \tr{}  \cgr[height=.7cm]{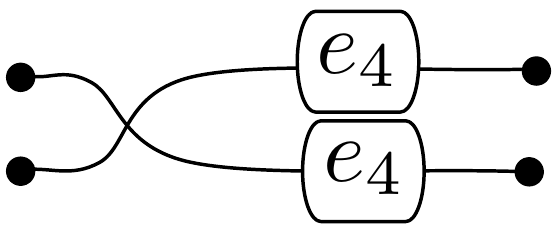} \tl{} 2.
         \end{eqnarray*}
We introduce another diagram $c \: 1\to 1$ in $\syntax{\Sigma}$ and its interpretation in $\FTerm{\Sigma}$:
                        \begin{equation*}\label{eq:counterexSoundness3}
                         \xymatrix@=10pt{
                        \lower7pt\hbox{$\includegraphics[height=.6cm]{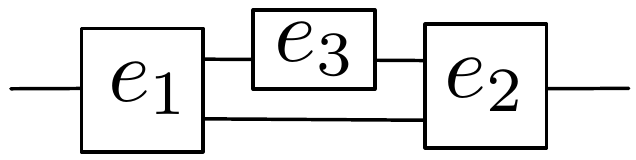}$}\quad \ar@{|->}[rr]^-{\SynToCsp{\cdot}} && \ \
                        1 \tr{} \lower7pt\hbox{$\includegraphics[height=.6cm]{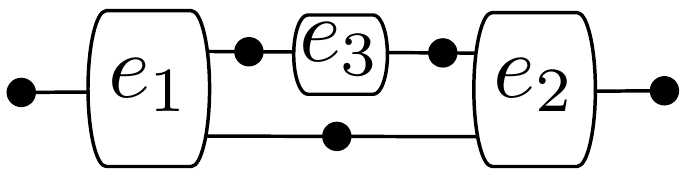}$} \tl{} 1.
                        }
          \end{equation*}
         Now, the left-hand side of the rule \eqref{eq:counterexSoundness1} cannot match in $c$. However, their interpretation yields a DPO rewriting step in $\FTerm{\Sigma}$ as below, where $f$ maps grey nodes to grey nodes.
         \begin{eqnarray*}
         \xymatrix@=15pt{
         {\cgr[height=.8cm]{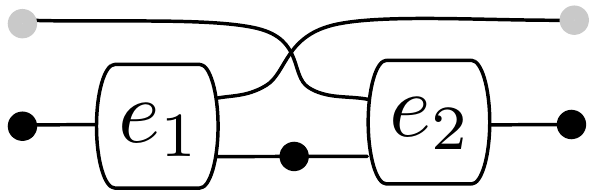} } \ar[d]_{f} & \ar[l] 2+2 \ar[d] \ar[r] & {\cgr[height=.8cm]{unsoundHypRroundSwap.pdf}} \ar[d]
          \\
         {\cgr[height=.95cm]{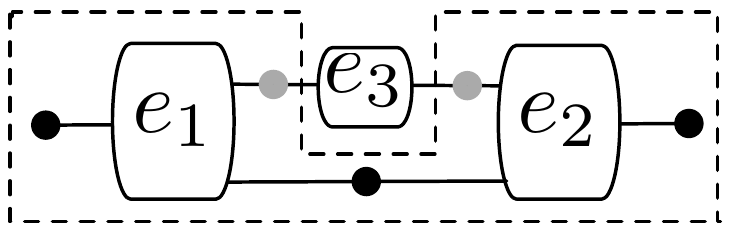}}  & \ar[l] {\cgr[height=.7cm]{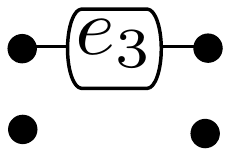}} \ar[r] & {\cgr[height=.45cm]{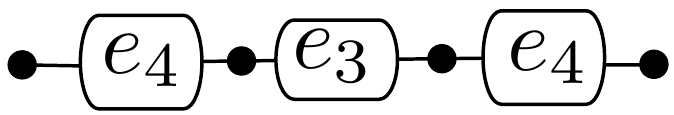}}
         }
         \end{eqnarray*}
         Observe that the leftmost pushout above \emph{is} a boundary complement: the input-ouput partition is correct. Still, mimicking this rewriting step in $\syntax{\Sigma}$ fails because the context does not fit the legal shape~\eqref{eq:goodcontext}. This is best evident by showing how matching in $\SynToCsp{c}$ appears under the isomorphism $\FTerm{\Sigma} \cong \syntax{\Sigma} + \frob$:
          \begin{eqnarray*}
          {\cgr[height=1.5cm]{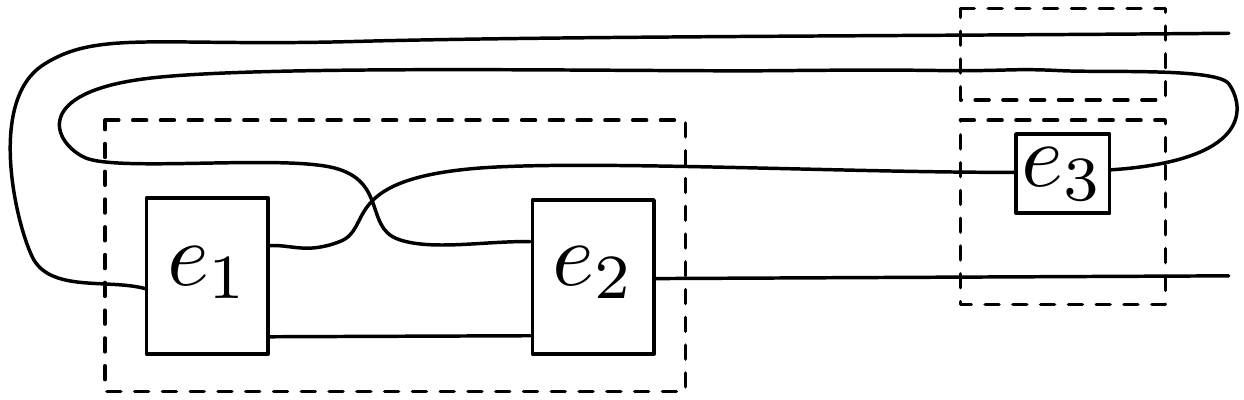}} \dfop {\cgr[height=1.2cm]{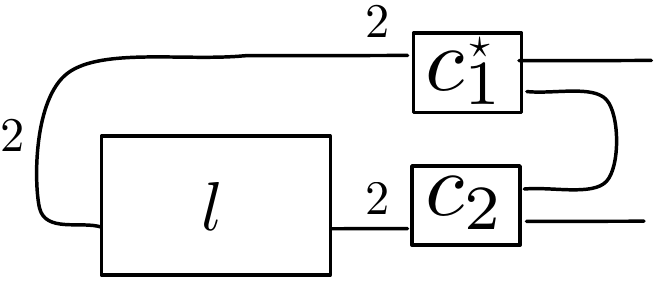}}
          \end{eqnarray*}
          The diagrams $l$, $c_1$ and $c_2$ on the right side are defined in terms of the dotted squares on the right. The difference with the prescription of~\eqref{eq:goodcontext} is apparent: the topmost output interface of $c_2$ connects to the bottommost input interface of $c_1$. For this reason, we cannot un-bend our context to obtain a legal matching in $\syntax{\Sigma}$.
\end{exm}

In order to rule out situations~\eqref{eq:counterexSoundness1} and complete the characterisation of legal contexts like~\eqref{eq:goodcontext}, we recover the notion of convexity (Def.~\ref{def:convexsubhyp}) to introduce \emph{convex} matchings. The intuition is that a convex matching cannot leave ``holes'' as $f$ does in Example~\ref{ex:unsound}.

\begin{defn}[Convex matching] \label{def:convexmatching}
We call $m : L \to G$ in $\Hyp{\Sigma}$ a \textit{convex matching} if it is mono and its image is convex.
\end{defn}

The importance of the above notion is best revealed by Lemma~\ref{lemma:convexfact} and Theorem~\ref{thm:charactImage}: given a convex matching, it is always possible to recover the appropriate context in $\syntax{\Sigma}$. We now combine the notions of boundary complement and of convex matching to tailor a family of DPO rewriting steps which only yield legal $\syntax{\Sigma}$-rewriting.

\begin{defn}[Convex DPO rewriting step] \label{def:rigidDpoRewriting} Let $\mathcal{R}$ be a DPO rewriting system. A \emph{convex} DPO rewriting step
\begin{eqnarray*}
\vcenter{\xymatrix@R=-2pt@C=22pt{ & D& \\ 0 \ar[ur] && \ar[ul]_{[q_1,q_2]} n+m}} & \rigidDPOstep{\mathcal{R}} & \vcenter{\xymatrix@R=-2pt@C=22pt{ & E & \\ 0 \ar[ur] && \ar[ul]_{[p_1,p_2]} n+m}}
\end{eqnarray*}
happens when there is a rule $L \xleftarrow{[a_1,a_2]} i+j \xrightarrow{[b_1,b_2]} R$ in $\mathcal{R}$ such that the following diagram in $\Hyp{\Sigma}$ commutes, the two squares are pushouts
\begin{equation}\label{eq:dpo3}
\raise25pt\hbox{$
\xymatrix@R=15pt@C=20pt{
L \ar[d]_{f}   &  i+j \ar[d]
 \ar@{}[dl]|(.8){\text{\large $\urcorner$}}
 \ar@{}[dr]|(.8){\text{\large $\ulcorner$}}
 \ar[l]_{[a_1,a_2]} \ar[r]^{[b_1,b_2]}  & R \ar[d] \\
 D &  C \ar[l] \ar[r]  & E \\
&  n+m \ar[u] \ar[ur]_{[p_1,p_2]}  \ar[ul]^{[q_1,q_2]}
}$}
\end{equation}
 and the following two restrictions hold:
 \begin{itemize}
 \item $f \: L \to G$ is a convex matching;
 \item in the leftmost pushout $i+j \to C$ is a boundary complement.
\end{itemize}
\end{defn}
Note that the relation $\rigidDPOstep{\mathcal{R}}$ is contained in $\DPOstep{\mathcal{R}}$ (Definition~\ref{defn:dpo}), the difference being that the leftmost pushout has to rest on a convex matching and a boundary complement.

We have now all the ingredients to prove the adequacy of convex DPO rewriting with respect to rewriting in $\syntax{\Sigma}$.

\begin{thm}\label{th:adequacyRigidSMT}Let $\mathcal{R}$ by any rewriting system  on $\syntax{\Sigma}$. Then, 
\begin{eqnarray*} d\Rightarrow_{\mathcal{R}}e & \Leftrightarrow & \Phi(\rewiring{d}) \rigidDPOstep{\Phi(\rewiring{\mathcal{R}})} \Phi(\rewiring{e}).\end{eqnarray*}
\end{thm}
\begin{proof} For the only if direction, by Theorem~\ref{thm:frobeniusrewriting} $d\Rightarrow_{\mathcal{R}}e$ implies $\Phi(\rewiring{d}) \DPOstep{\Phi(\rewiring{\mathcal{R}})} \Phi(\rewiring{e})$. It is not hard to verify that the argument actually constructs a \emph{convex} DPO rewriting, thus yielding the desired statement. The reader may consult Appendix~\ref{app:proofs} for details.

We now turn to the converse direction. Let 
\begin{eqnarray*}
\Phi(\rewiring{d}) \dfop \vcenter{\xymatrix@R=5pt@C=3pt{ & D& \\ 0 \ar[ur] && \ar[ul]_{[q_1,q_2]} n+m}} & & \Phi(\rewiring{e}) \dfop \vcenter{\xymatrix@R=5pt@C=3pt{ & E & \\ 0 \ar[ur] && \ar[ul]_{[p_1,p_2]} n+m}}.
\end{eqnarray*}
Our assumption gives us a diagram as in~\eqref{eq:dpo3}, with application of a rule $\rrule{\Phi(\rewiring{l})}{\Phi(\rewiring{r})}$ in $\Phi(\rewiring{\mathcal{R}})$. We now want to show that $d\Rightarrow_{\mathcal{R}}e $ with rule $\rrule{l}{r}$, say of type $(i,j)$. Now, because $n \tr{q_1} D \tl{q_2} m = \SynToCsp{d}$, it is monogamous directed acyclic by Theorem~\ref{thm:charactImage}. Since the matching $f \: L \to D$  in~\eqref{eq:dpo3} is convex, Lemma~\ref{lemma:convexfact} yields a decomposition of $n \tr{q_1} D \tl{q_2} m$ in terms of monogamous directed acyclic cospans:
\begin{equation*}
\left( n \tr{} C_1 \tl{} i\!+\!k \right) \poi\!\!\!
\begin{array}{cc}
\left( k \tr{id} k \tl{id} k \right)\\
\tns \\
\left( i \tr{} L \tl{} j \right)
\end{array}
\!\!\!\poi
\left( j\!+\!k\tr{} C_2 \tl{} m \right).
\end{equation*}
Applying again Theorem~\ref{thm:charactImage} we obtain $c_1$, $c_2$ in $\syntax{\Sigma}$ such that
\begin{align*}
\SynToCsp{c_1} = n\tr{} C_1 \tl{} i+k \qquad \SynToCsp{c_2} = j+k\tr{} C_2 \tl{} m.
\end{align*}
By functoriality of $\SynToCsp{\cdot}$, $\SynToCsp{d} = \SynToCsp{c_1 \poi (id \tns l) \poi c_2} $ and, since $\SynToCsp{\cdot}$ is a faithful prop morphism, $d = c_1 \poi (id \tns l) \poi c_2$. Thus we can apply the rule $\rrule{l}{r}$ on $e$, which yields $e = c_1 \poi (id \tns r) \poi c_2$ such that $d \Ra_{\mathcal{R}} e$. We can conclude that $\SynToCsp{e} = n\tr{p_1} E \tl{p_2} m$ because boundary complements are unique (Proposition~\ref{thm:uniquenessBoundaryCompl}).
\end{proof}

\section{Case study: non-commutative bialgebras}\label{sec:terminationHopf}

\begin{figure*}\caption{DPO Rewriting system $\Phi(\rewiring{\mathcal{R}_{\scriptscriptstyle \NCB}})$ --- \emph{cf.} Example~\ref{ex:ncbialgebrasrewriting}.} \label{fig:hopf}
  \[
  \begin{array}{rrcl} 
      \BA{1} := & \!\!\!\!\cgr{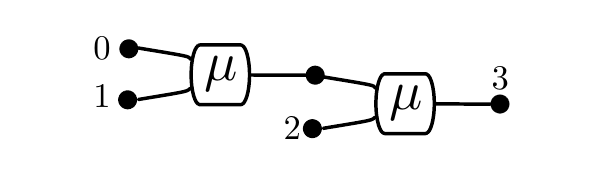}\!\!\!\! & \leftarrow 4 \rightarrow & \!\!\!\!\cgr{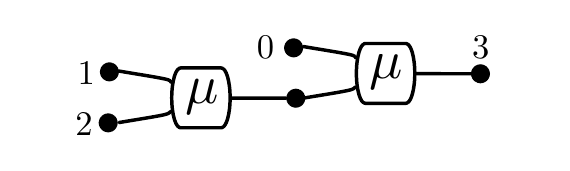}
      \\
      \BA{2} := & \!\!\!\!\cgr{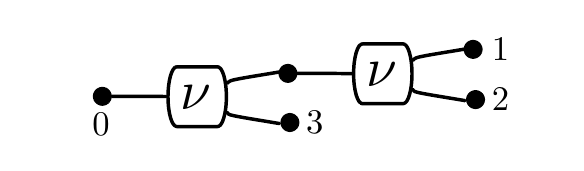}\!\!\!\! & \leftarrow 4 \rightarrow & \!\!\!\!\cgr{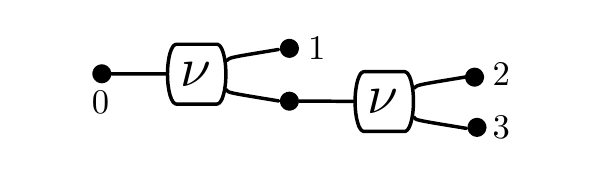}\!\!\!\!
      \\
      \BA{3} := & \!\!\!\!\cgr{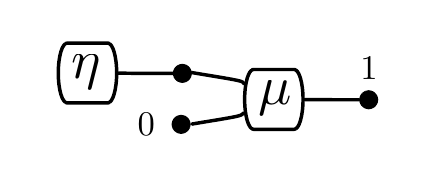}\!\!\!\! & \leftarrow 2 \rightarrow & \!\!\!\!\cgr{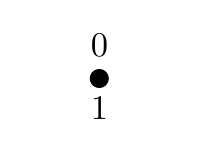}\!\!\!\!
      \\
      \BA{4} := & \!\!\!\!\cgr{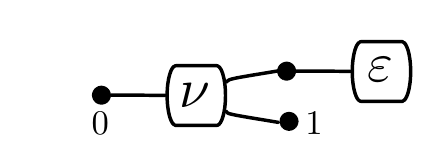}\!\!\!\! & \leftarrow 2 \rightarrow & \!\!\!\!\cgr{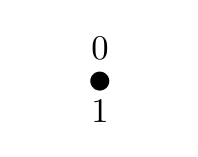}\!\!\!\!
      \\
      \BA{5} := & \!\!\!\!\cgr{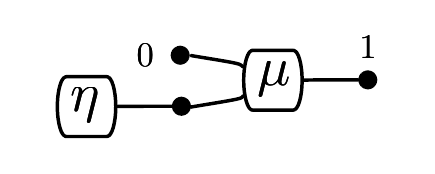}\!\!\!\! & \leftarrow 2 \rightarrow & \!\!\!\!\cgr{unitR.pdf}\!\!\!\!
      \\
      \BA{6} := &  \!\!\!\!\cgr{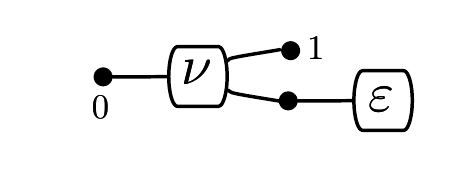}\!\!\!\! & \leftarrow 2 \rightarrow & \!\!\!\!\cgr{counitR.pdf}\!\!\!\!
      \\
      \BA{7} := & \!\!\!\!\cgr{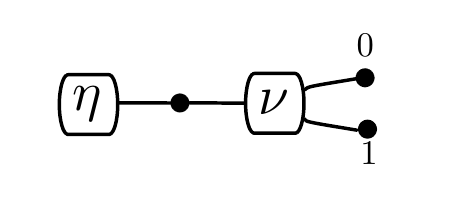}\!\!\!\! & \leftarrow 2 \rightarrow & \!\!\!\!\cgr{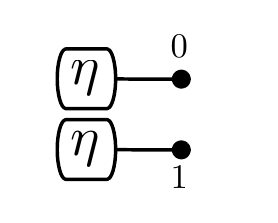}\!\!\!\!
      \\
      \BA{8} := & \!\!\!\!\cgr{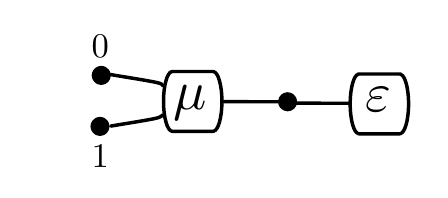}\!\!\!\! & \leftarrow 2 \rightarrow & \!\!\!\!\cgr{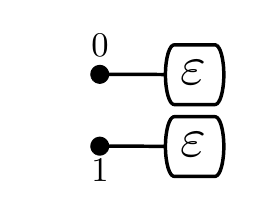}\!\!\!\!
      \\
      \BA{9} := & \!\!\!\!\cgr{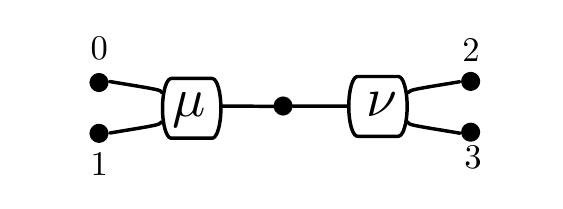}\!\!\!\! & \leftarrow 4 \rightarrow & \!\!\!\!\cgr{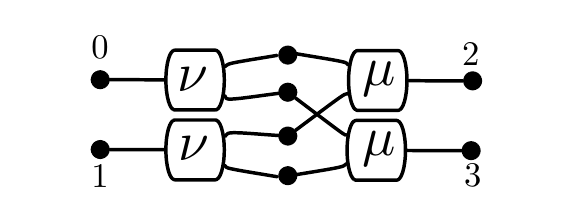}\!\!\!\!
      \\
      \BA{10} := & \!\!\!\!\cgr{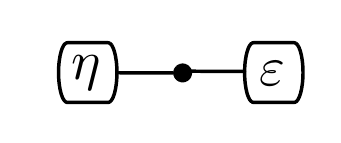}\!\!\!\! & \leftarrow 0 \rightarrow &  {  }
    \end{array}
  \]
\end{figure*}

Thanks to Theorem~\ref{th:adequacyRigidSMT}, we are now in position to recover Example~\ref{ex:ncbialgebrasrewriting} and study the DPO rewriting theory of non-commutative bialgebras (Example~\ref{exm:props}.\ref{ex:ncbialgebras}), shown in Figure~\ref{fig:hopf}. As a proof of concept, we focus on proving \emph{termination} for this system. First, we construct a metric based on two kinds of paths and a weight.
\begin{itemize}
  \item A \textit{U-path} is a path $p$ from an input or an $\eta$-hyperedge to an output or an $\epsilon$-hyperedge.
  \item An \textit{M-path} is a path from a $\mu$-hyperedge to a $\nu$-hyperedge.
\end{itemize}

A $\mu$-tree with root $x$ is a maximal tree of $\mu$-hyperedges with output $x$. Similarly, a $\nu$-tree with root $x$ is a maximal tree of $\nu$-hyperedges with input $x$.

For a $\mu$-hyperedge $h$, let the L-weight $\ell(h)$ be the size of the $\mu$-tree whose root is the first input of $\mu$. Similarly, for a $\nu$-hyperedge, let $\ell(h)$ be the size of the $\nu$-tree whose root is the first output of $h$. Let $\ell(h) = 0$ otherwise and:
\[
G \precL H \iff \sum_{h \in E_G} \ell(h) \leq \sum_{h \in E_H} \ell(h).
\]
Next, define orders \precU, \precM, \precmu, \precnu based on counting the number of U-paths, M-paths, $\mu$-hyperedges, and $\nu$-hyperedges, respectively. Armed with the five orders, we define the following lexicographic ordering that combines them as its components:
\begin{equation}\label{eq:lex-metric}
  \preceq \ :=\  \textrm{lex}(\precU, \precM, \precmu, \precnu, \precL)
\end{equation}
Clearly, $\preceq$ is well-founded, thus we can conclude as follows.

\begin{thm}
  $\BA{}$ is strictly decreasing in $\preceq$, thus it terminates.
\end{thm}

\begin{proof}
  We argue rule-by-rule, showing that each is strictly decreasing in one of the orders from \eqref{eq:lex-metric}, and non-increasing in every order that is prior in the lexicographic ordering.

  Since every rule \BA{j} has a unique path from every input to every output on the LHS and RHS, applications of these rules have no effect on paths which start and finish outside of their image. Hence, for each rule, we only need to consider paths which start or terminate in the image of the LHS.

  \BA{1} has no effect on $\eta$ or $\epsilon$ hypereges, hence no effect on \precU. No M-path can terminate in \BA{1} and any M-path originating on \BA{1} must exit through the unique output. Since there are precisely two $\mu$-hyperedges in the LHS and RHS, there is a 1-to-1 correspondence between M-paths before and after applying the rule. \BA{1} also leaves the number of $\mu$ and $\nu$ hyperedges fixed, so we only need to show it strictly decreases \precL. Applying the rule has no effect on the L-weight of any $\mu$-hyperedges outside of the image of the LHS. Suppose there are $\mu$-trees of size $a, b, c$ connected to inputs $0, 1, 2$ of the LHS, respectively. Then, the L-weight of the two $\mu$-hyperedges on the LHS are $a$ and $a + b + 1$, whereas on the LHS they are $a$ and $b$. Hence \precL is strictly decreased. \BA{2} follows via a symmetric argument.

  Since \BA{3}--\BA{6} and \BA{10} remove $\eta$- and $\epsilon$-hyperedges from the hypergraph, they will strictly decrease the number of U-paths.

  For \BA{7}, no U-path can terminate in the LHS, and any U-path starting in LHS must exit through one of the two outputs. Hence it corresponds to a unique U-path exiting the RHS. M-paths are unaffected, as is the number of $\mu$-hyperedges. However, the number of $\nu$-hyperedges is strictly decreased, so \BA{7} strictly decreases \precnu. The argument for \BA{8} is again symmetric.

  \BA{9} has no $\eta$ or $\epsilon$-hyperedges in its LHS or RHS, so it leaves the number of U-paths fixed. Consider an M-path which enters the LHS from the left. It must enter either from input $0$ or input $1$, hence it corresponds to a unique M-path entering the RHS. We can argue similarly for M-paths exiting on the right. Hence, the only M-path left to consider is the one from the $\mu$-hyperedge to the $\nu$-hyperedge in the LHS, which is eliminated. Thus \BA{9} strictly reduces \precM.
\end{proof}

Hence, the combinatorial presentation gives a short and simple proof of termination. We conjecture that one can construct a similarly simple proof of confluence, via analysis of critical pairs. There is an existing notion of critical pair in the DPO literature, of which \BA{} rewrite system has 24 (all of which are `joinable'). There are, however, subtleties arising in critical pair analysis for DPO rewrite systems~\cite{Plump1993} and general rewriting for SMCs~\cite{Lafont2003}. We leave development of a comprehensive theoretical framework for critical pair analysis on convex DPO rewrite systems as future work.


\section{Conclusions and future work}

Our approach centered on building a bridge between
algebraic and combinatorial descriptions of string diagrams, in the form of the isomorphism $\syntax{\Sigma}+\frob \cong \FTerm\Sigma$. We used it to spell out two fundamental ways in which rewriting in symmetric monoidal categories and DPO hypergraph rewriting are compatible.


The first is that we can use DPO to rewrite SMCs generated by theories containing a special Frobenius structure. Theories of that kind are appearing in different research threads: they play a role in the compositional study of Petri nets ~\cite{Bruni2011,Sobocinski2014,Rathke2014a}, in the algebra of concurrent connectors~\cite{Bruni2006}, in signal flow diagrams~\cite{Bonchi2014b,Bonchi2015,BaezErbele-CategoriesInControl,Fong2015} and in the ZX-calculus~\cite{Coecke2008, Coecke2012}, just to mention a few examples. Our approach enables the use of the technology that has been developed for DPO in the last decades in those diverse areas.

The second approach -- convex DPO rewriting -- allows to rewrite in categories presented by arbitrary symmetric monoidal theories. Our results lay the theoretical foundations for the development of comprehensive tool support for reasoning about rewriting in SMCs. Emerging tools such as Globular~\cite{Globular} do not factor in the symmetric structure, whereas Quantomatic~\cite{Kissinger_quantomatic} relies on additional structure, requiring for example that the ambient category be traced. Our theory overcomes these limitations and can be effectively implemented: matching is efficient with respect to the size of the target hypergraph and the two conditions in the definition of a convex rewriting step -- convexity and boundary complement -- can be be automatically checked. Furthermore, under the encoding of hypergraphs as string graphs (the underlying theory used by Quantomatic), it is now possible to extend Quantomatic to support rewriting in a generic SMC. We leave the algorithmic details as future work.

Just as in term rewriting, commutative operators pose a interesting problem: naively encoding commutativity in a rule immediately yields a non-terminating system. We leave this issue as future work, but preliminary results suggest that the commutativity of some operators can be effectively encoded in the graphical structure.

From a theoretical point of view, we intend to build on this work to investigate the relationship between distributive laws of PROPs~\cite{Lack2004a} and rewriting. Compatibility conditions expressed by a distributive law come with an orientation, making natural to think of them as rewriting rules. Just as in our approach, these conditions are defined modulo the laws of SMCs: it thus appears promising to explore the correspondence between the two perspectives.

\bibliographystyle{akbib}
\bibliography{catBib3}
\newpage

\appendix
\section{Omitted proofs}\label{app:proofs}

\paragraph{Preliminaries on hypergraphs} The following terminology will be used throughout the proofs of this appendix. Given an hypergraph $G$, we write $N_G$ for its set of nodes and $E_G$ for the set of hyperedges. Given an hyperedge $h$ with $k$ (ordered) sources and $l$ (ordered) targets, we write $s_j(h)$ for the $j$-th source node of $h$ and $t_j(h)$ for the $j$-th target node of $h$.

For a pair of hyperedges $h,h'$ such that for some $i,j$ and node $v$, $t_i(h) = v = s_j(h)$, then $h$ is called a \textit{predecessor} of $h'$ and $h'$ a \textit{successor} of $h$.

    For a hypergraph $G$ and hyperedges $h,h'$, a \textit{directed path} $p$ from $h$ to $h'$ is a sequence of hyperedges $[h_1, \ldots, h_n]$ such that $h_1 = h$, $h_n = h'$, and for $i < n$, $h_{i+1}$ is a successor of $h_i$. We say $p$ starts at a subgraph $H$ if $H$ contains a node in the source of $h$, and terminates at a subgraph $H'$ if $H'$ contains a node in the target of $h'$. We say $p$ starts at a subgraph $H$ if there exists some $i$ such that $s_i(h) \in H$, and terminates at a subgraph $H'$ if there exists some $j$ where $t_j(h') \in H'$.

Note that, by regarding nodes as single-node subgraphs, it clearly makes sense to talk about directed paths from/to nodes as well.

\begin{proof}[Proof of Lemma~\ref{lemma:convexfact}]
    Let $C_1$ be the smallest sub-hypergraph containing the inputs of $G$ and every hyperedge $h$ which is not in $L$, but has a path to it. Let $C_2$ then be the smallest sub-hypergraph containing the outputs of $G$ such that $C_1 \cup L \cup C_2 = G$. By construction, $C_1$ and $L$ share no hyperedges. If $C_2$ shared a hyperedge with $C_1$ or $L$, then there would exist a smaller $C_2'$ such that $C_1 \cup L \cup C_2' = G$, so $C_2$ shares no hyperedges with either $C_1$ or $L$. Hence, the three sub-hypergraphs only overlap on nodes. Now let:
    \begin{align*}
        i & := N_{C_1} \cap N_L \\
        j & := N_{C_2} \cap N_L \\
        k & := (N_{C_1} \cap N_{C_2}) \backslash N_L
    \end{align*}
    Since $p$ is disjoint from $m'$ and $n'$, it follows that:
    \[
        k + i = N_{C_1} \cap (N_L \cup N_{C_2})
        \qquad
        k + j = N_{C_2} \cap (N_L \cup N_{C_1})
    \]
    Now, define the following cospans, where arrows are all inclusions:
    \[ n \rightarrow C_1 \leftarrow k + i \]
    \[ i \rightarrow L \leftarrow j \]
    \[ k + j \rightarrow C_2 \leftarrow m \]
    Then \eqref{eq:snd-factor} is computed as the colimit of the following diagram:
    \[ n \rightarrow C_1 \leftarrow k + i
         \rightarrow k + L \leftarrow k + j
         \rightarrow C_2 \leftarrow m \]
    The two spans identify precisely those nodes from $G$ which occur in more than one sub-hypergraph, so this amounts to simply taking the union:
    \[ n \rightarrow C_1 \cup L \cup C_2 \leftarrow m \ =\ n \rightarrow G \leftarrow m \]

    $C_1, C_2$, and $L$ are all directed acyclic because $G$ is, so it only remains to show each of these cospans is monogamous. $C_1$ and $C_2$ follow straightforwardly from the observation that $C_1$ is closed under predecessors and $C_2$ under successors.


    The interesting case is $L$, which relies on convexity. Suppose $v$ does not have a in-hyperedge in $L$. Then either $v$ is an input or there exists a hyperedge with a path to $v$. One of these two is true precisely when $v \in i$. Suppose $v$ does not have an out-hyperedge in $L$. Then, either $v$ is an ouput or it has an out-hyperedge in $C_1$ or $C_2$. But, if it has an out-hyperedge $h$ in $C_1$, then there is a path from $v$ to another node $v'$, going through $h$. But by convexity, $h$ must then be in $L$, which is a contradiction. Hence $v \in C_2$, which is true if and only if $v \in j$.
\end{proof}

We now introduce the necessary ingredients for the proof of Proposition~\ref{prop:faithful}. The initial object in the category $\PROP$ is the PROP $\PERM$ of permutations: arrows $n \to m$ in $\PERM$ exist only if $m = n$, in which case they are the permutations of the finite ordinal with $n$ elements.
Given a PROP $\X$, the unique PROP morphism $\PERM \to\X$ need not be, in general, faithful; an extreme example is the terminal prop $\mathbf{1}$ where all hom-sets are singletons.
We will say that $\X$ is \emph{faithful} if $\PERM \to\X$ is faithful. Roughly speaking, faithfulness says that no two different permutations are identified. Proposition~\ref{prop:faithful} will follow as corollary of the next lemma about faithful PROPs.

\begin{lem}\label{lem:faithfulCoproducts}
Coproducts of faithful PROPs are faithful and have faithful coproduct injections.
\end{lem}
\begin{proof}
We use the fact that the category of PROPs can be identified with a full subcategory of the coslice $\PERM /\PRO$, see~e.g.\cite[Prop.~2.8]{ZanasiThesis}. Here $\PRO$ is the category of PROs, with: (i) object strict monoidal categories that have the natural numbers as the set of objects, where monoidal product on objects is addition, and (ii) arrows strict monoidal identity-on-objects functors. Since $\PERM$ is symmetric monoidal, $\PROP$ is then the full subcategory of $\PERM /\PRO$ with objects the \emph{symmetric} monoidal functors.

To calculate the coproduct $\X+\Y$ of two faithful PROPs $\X$ and $\Y$ in $\PROP$ is to take the following pushout in $\PRO$.
\[
\xymatrix@R=10pt{
& \PERM \ar[dl]_x \ar[dr]^y \\
\X \ar[dr]_-{i_1} & & \Y \ar[dl]^-{i_2} \\
& \X+\Y
}
\]
Since $x$ and $y$ are symmetric monoidal, it follows that $i_1$ and $i_2$ are symmetric monoidal~\cite[Prop.~2.10]{ZanasiThesis}. It thus remains to show that they are faithful.

Here we use a result of MacDonald and Scull~\cite[Th.~3.3]{MacDonald2009} about amalgamation in the category of categories $\mathbf{Cat}$. Indeed, since $x$ and $y$ are identity-on-objects, $i_1$ and $i_2$ are also identity-on-objects. It suffices to show that $x$ and $y$ satisfy the so-called 3-for-2 property: for $x$, this requires that given $h=g\circ f$ in $\mathbb{X}$, if any two of $f,g,h$ are in the image of $x$ then so is the third.
It is enough to consider $x$ since the argument for $y$ is identical. And here the 3-for-2 property trivially holds, since every arrow of $\PERM$ is an isomorphism.
\end{proof}

\begin{proof}[Proof of Proposition~\ref{prop:faithful}]
Immediate from Lemma~\ref{lem:faithfulCoproducts} since the top row of~\eqref{eq:fterm} is a coproduct diagram in $\mathsf{PROP}$, and $\syntax{\Sigma}$, $\frob$ are both faithful.
\end{proof}

\begin{figure}\caption{Pushout of diagram~\eqref{eq:decomp} \label{fig:bigpushout}}
\centering
    \[
\xymatrix@R=15pt@C=10pt{
n \ar[dr]_{f} & & N+\tilde{n} \ar@{}[dd]|(.8){\wedge}
\ar[dl]_{[id\ j]} \ar[dr]^{id\oplus i_1}
& & {N+\tilde{m}} \ar@{}[dd]|(.8){\wedge}
\ar[dl]_{id\oplus i_2} \ar[dr]^{[id\ p]} & & m. \ar[dl]^{g}
\\
&  N \ar[dr]_{i_1} & & N + \tilde{n} + \tilde{m} \ar@{}[dd]|(.8){\wedge}
 \ar[dl]_{[id \ j]\oplus id }
\ar[dr]^-{([id\ p] \oplus id)\sigma_{\tilde{n},\tilde{m}}}
& & N \ar[dl]^{i_1}
\\
& & N+\tilde{m} \ar[dr]_{[id\ p]}& & N + \tilde{n} \ar[dl]^{[id\ j]}
\\
& & & N
}
\]
\end{figure}

\begin{proof}[Proof of Proposition~\ref{thm:uniquenessBoundaryCompl}]
    Since $\Hyp{\Sigma}$ is a presheaf category, a pushout of hypergraphs consists of pushouts on the underlying sets of nodes and hyperedges and an appropriate choice of source and target maps. Hence the underlying sets of $L^\perp$ must give pushout complements in $\F$:
    \[
    \xymatrix{
    N_L \ar[d]_{f_V} &  i + j \ar[l]_{a_V} \ar[d]^{c_V} \ar@{}[dl]|(.8){\text{\large $\urcorner$}}  \\
    N_G & N_{L^\perp} \ar[l]_{g_V}  \\
    }
    \qquad
    \xymatrix{
    E_L \ar[d]_{f_E} & 0 \ar[l] \ar[d] \ar@{}[dl]|(.8){\text{\large $\urcorner$}}  \\
   E_G & E_{L^\perp} \ar[l]_{g_E}  \\
    }
    \]
    Since $i + j$ is discrete, $E_G$ is a disjoint union, so $E_{L^\perp}$ must be $E_G \backslash E_L$. Since these are diagrams in $\F$, and $c,f$ are mono, we can rewrite the left pushout square as:
    \[
    \xymatrix{
    l + x \ar[d]_{f_V} &  i + j \ar[l]_{a_V} \ar[d]^{c_V} \ar@{}[dl]|(.8){\text{\large $\urcorner$}}  \\
    l + x + y & i + j + z \ar[l]_{g_V}  \\
    }
    \]
    where the two downward arrows are coproduct injections and $l$ is the image of $a_V$. One can easy verify that $l + x + z$ also gives a pushout for the given span, so we obtain a commuting isomorphism $l + x + y \cong l + x + z$, from which we conclude $y \cong z$ and that, up to isomorphism, the pushout complement on nodes must be:
    \[
    \xymatrix{
   l + x \ar[d]_{f_V}  &  i + j \ar[l]_{a_V} \ar[d]^{c_V} \ar@{}[dl]|(.8){\text{\large $\urcorner$}}  \\
    l + x + y  &  i + j + y \ar[l]^{g_V} \\
    }
    \]
    from whence it follows that $g_V = a_V + id_n$.

\medskip
So far, we have proved that the set of nodes $N_{L^\perp}$ and $E_{L^\perp}$ are defined uniquely by the property of being a pushout complement. The only thing that remains is the source $s_p\: E_{L^\perp} \to N_{L^\perp}$ and target $t_p\: E_{L^\perp} \to N_{L^\perp}$ maps.
Since $g$ is a homomorphism, we have that for all hyperedges $h$:
    \[ g_V(s_p(h)) = s_p(g_E(h)) \ \implies\ s_p(h) \in g_V^{-1}(s_p(g_E(h))) \]
    Since $g_V$ is of the form $[(a_1)_V, (a_2)_V] + 1_n$, where $a_1$ and $a_2$ are mono, the inverse image $g_V^{-1}(s_p(g_E(h)))$ contains at most two elements. In the case where it has 1 element, $s_p$ is uniquely fixed, so consider when it has two. It must then be the case that:
    \[ g_V^{-1}(s_p(g_E(h))) = \{ v_1 \in i, v_2 \in j \} \]

But monogamicity of \eqref{eq:boundary} says that the image of $i$ in $L^\perp$ cannot be the source of any hyperedge. Therefore, it must be $s_p(h) = v_2$. Similarly if:
    \[ g_V^{-1}(t_p(g_E(h))) = \{ v_1 \in i, v_2 \in j \} \]
    then $t_p(h)$ must be $v_1$. Since there is at most one choice of source and target maps for $L^\perp$ making $g$ a homomorphism, $L^\perp$ must be unique.
\end{proof} 

\begin{proof}[Proof of Theorem~\ref{th:adequacyRigidSMT}]
We report here the argument for the only if direction of the statement, which is not detailed in the main text. Our assumption gives that
\begin{eqnarray*}
 \cgr{circuitd.pdf} & =& \cgr{factl.pdf} \\
 \cgr{circuite.pdf} & =& \cgr{factr.pdf}.
\end{eqnarray*}
whence the following equalities hold in $\syntax{\Sigma} + \frob$, see~\eqref{eq:rewd}-\eqref{eq:rewe}.
\begin{eqnarray}
  \cgr{rewd1.pdf} & =& \cgr{rewd4.pdf} \label{eq:rewcompld}\\
    \cgr{rewe1.pdf} & =& \cgr{rewe4.pdf} \label{eq:rewcomple}
\end{eqnarray}
We now define
\begin{eqnarray*}
 \left( 0 \tr{} D \tl{[q_1,q_2]} n+m\right) & \df & \Phi(\rewiring{d}) = \Phi \left( \cgr[height=.6cm]{rewd1.pdf} \right) \\
 \left( 0 \tr{} E \tl{[p_1,p_2]} n+m\right) & \df & \Phi(\rewiring{e}) = \Phi \left( \cgr[height=.6cm]{rewe1.pdf} \right) \\
 \left( 0 \tr{} L \tl{[a_1,a_2]} i+j\right) & \df & \Phi(\rewiring{l}) = \Phi \left( \cgr[height=.6cm]{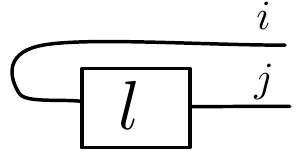} \right) \\
 \left( 0 \tr{} R \tl{[b_1,b_2]} i+j\right) & \df & \Phi(\rewiring{r}) = \Phi \left( \cgr[height=.6cm]{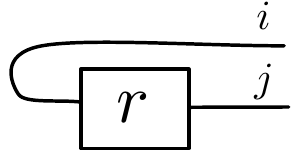} \right) \\
 \left( i+j \tr{} C \tl{} n+m \right) & \df & \Phi \left( \cgr[height=1.1cm]{SplittedContext.pdf} \right). \\
\end{eqnarray*}
By these definitions and \eqref{eq:rewcompld}-\eqref{eq:rewcomple} it follows that
\begin{eqnarray*}
\left( 0 \tr{} D \tl{} n+m \right) = \left( 0 \tr{} L \tl{} i+j\right) \poi \left( i+j \tr{} C \tl{} n+m\right) \\
\left( 0 \tr{} E \tl{} n+m \right) = \left( 0 \tr{} R \tl{} i+j\right) \poi \left( i+j \tr{} C \tl{} n+m\right).
\end{eqnarray*}
Since composition of cospans is defined by pushout, we have a commutative diagram with two pushouts as in~\eqref{eq:dpo3}. It remains to check that the matching $L \to D$ is convex and that $C$ is a boundary complement: these conditions can be verified by definition of the involved components. Therefore, $\Phi(\rewiring{d}) \rigidDPOstep{\Phi(\rewiring{\mathcal{R}})} \Phi(\rewiring{e})$ by application of the rule $\rrule{\Phi(\rewiring{l})}{\Phi(\rewiring{r})}$.
\end{proof}

\end{document}